\documentclass[12pt]{amsart}
\usepackage{amsmath,amsthm,bm,amssymb,wasysym}
\usepackage[margin=1in]{geometry}
\usepackage{graphicx}
\usepackage{tikz}
\usepackage{ifthen}
\usepackage[usenames,dvipsnames]{pstricks}
\usepackage{epsfig}
\usepackage{pst-grad} % For gradients
\usepackage{pst-plot} % For axes
\usepackage{enumitem}
\usepackage{algorithm}
\usepackage{algorithmic}
\usepackage{wrapfig}
\usepackage{color}
\def\calN{{\mathcal N}}
\def\calD{{\mathcal D}}

\def\boldf{{\mathbf f}}

\newtheorem{thm}{Theorem}[section]

\newcommand{\blockmatrix}[9]{
  \draw[draw=#4,fill=#5] (0,0) rectangle( #1,#2);
  \ifthenelse{\equal{#6}{true}}
  {
    \draw[draw=#7,fill=#8] (0,#2) -- (#9,#2) -- ( #1,#9) -- ( #1,0) -- ( #1 - #9,0) -- (0,#2 -#9) -- cycle;
  }
  {}
  \draw ( #1/2, #2/2) node { #3};
}

\newcommand{\mblockmatrix}[4][none]{
  \begin{tikzpicture}
  \ifthenelse{\equal{#1}{none}}
  {
    \blockmatrix{#2}{#3}{#4}{none}{none}{false}{none}{none}{0.0}
  }
  {
    \definecolor{fillcolor}{rgb}{#1}
    \blockmatrix{#2}{#3}{#4}{none}{fillcolor}{false}{none}{none}{0.0}
  }
  \end{tikzpicture}%this comment is necessary
}

\newcommand{\fblockmatrix}[4][none]{
  \begin{tikzpicture}
  \ifthenelse{\equal{#1}{none}}
  {
    \blockmatrix{#2}{#3}{#4}{black}{none}{false}{none}{none}{0.0}
  }
  {
    \definecolor{fillcolor}{rgb}{#1}
    \blockmatrix{#2}{#3}{#4}{black}{fillcolor}{false}{none}{none}{0.0}
  }
  \end{tikzpicture}%this comment is necessary
}

\newcommand{\dblockmatrix}[4][none]{
  \begin{tikzpicture}
  \ifthenelse{\equal{#1}{none}}
  {
    \blockmatrix{#2}{#3}{#4}{black}{none}{true}{black}{none}{0.35cm}
  }
  {
    \definecolor{fillcolor}{rgb}{#1}
    \blockmatrix{#2}{#3}{#4}{black}{none}{true}{black}{fillcolor}{0.35cm}
  }
  \end{tikzpicture}%this comment is necessary
}

\newcommand{\Imatrix}[4][none]{
  \begin{tikzpicture}
  \ifthenelse{\equal{#1}{none}}
  {
    \blockmatrix{#2}{#3}{#4}{black}{none}{true}{black}{none}{0.1cm}
  }
  {
    \definecolor{fillcolor}{rgb}{#1}
    \blockmatrix{#2}{#3}{#4}{black}{none}{true}{black}{fillcolor}{0.1cm}
  }
  \end{tikzpicture}%this comment is necessary
}

\newcommand{\diagonalblockmatrix}[5][none]{
  \begin{tikzpicture} 

  \ifthenelse{\equal{#1}{none}}
  {
    \blockmatrix{#2}{#3}{#4}{black}{none}{true}{black}{none}{#5}
  }
  {
    \definecolor{fillcolor}{rgb}{#1}
    \blockmatrix{#2}{#3}{#4}{black}{none}{true}{black}{fillcolor}{#5}
  }

  \end{tikzpicture}%necessary comment
}

\newcommand{\valignbox}[1]{
  \vtop{\null\hbox{#1}}% necessary comment
}

\newenvironment{blockmatrixtabular}
{% necessary comment
  \begin{tabular}{
  @{}l@{}l@{}l@{}l@{}l@{}l@{}l@{}l@{}l@{}l@{}l@{}l@{}l@{}l@{}l@{}l@{}l@{}l@{}l
  @{}l@{}l@{}l@{}l@{}l@{}l@{}l@{}l@{}l@{}l@{}l@{}l@{}l@{}l@{}l@{}l@{}l@{}l@{}l
  @{}l@{}l@{}l@{}l@{}l@{}l@{}l@{}l@{}l@{}l@{}l@{}l@{}l@{}l@{}l@{}l@{}l@{}l@{}l
  @{}
  }
}
{
  \end{tabular}%necessary comment
}

\usepackage{bbm}
\newcommand{\R}{\mathbbm{R}}

\newcommand{\bl}{\pmb{\lambda}}
\newcommand{\x}{\underline{x}}
\newcommand{\y}{\underline{y}}
\newcommand{\ppx}[1]{\frac{\partial}{\partial #1}}

\newcommand{\rrbm}{R$^2$BFM}

\DeclareMathOperator*{\argmax}{argmax}
\DeclareMathOperator*{\argmin}{arg\,min}

\theoremstyle{remark}
\newtheorem*{remark}{Remark}

\begin{document}
\title[ ]{A Reduced Radial Basis Function Method for Partial Differential Equations 
on irregular domains}
\author{Yanlai Chen$^{1,2}$, 
Sigal Gottlieb$^{1,3}$, Alfa Heryudono$^{1,3,4}$, Akil Narayan$^{1,4}$}

\address{1 - Department of Mathematics, University of Massachusetts Dartmouth, 285 Old Westport Road, North Dartmouth, MA 02747, USA.}
\address{2 - The research of this author was partially supported by National Science Foundation grant DMS-1216928.}
\address{3 - The research of this author was partially supported by AFOSR grant FA9550-09-1-0208.}
\address{4 - The research of this author was partially supported by National Science Foundation grant DMS-1318427.}

\begin{abstract}
We propose and test the first Reduced Radial Basis Function Method (R$^2$BFM) for solving parametric partial differential equations on irregular domains. The two major ingredients are a stable Radial Basis Function (RBF) solver that has an optimized set of centers chosen through a reduced-basis-type greedy algorithm, and a collocation-based model reduction approach that systematically generates a reduced-order approximation whose dimension is orders of magnitude smaller than the total number of RBF centers. The resulting algorithm is efficient and accurate as demonstrated through two- and three-dimensional test problems.
\end{abstract}
\maketitle

% Introduction
\section{Introduction}

Parameterized systems are common in science and engineering, and a common situation in multi-query contexts is the need to accurately solve these systems for a variety of different parameter values. This requires a large number of repeated and expensive simulations, frequently rendering the total computational cost prohibitive. To overcome this obstacle while maintaining accurate numerical solutions in real time, the {\em reduced basis method} (RBM) was developed \cite{Barrett_Reddien,Noor_Peters,Peterson,Prudhomme_Rovas_Veroy_Maday_Patera_Turinici}. RBMs split the solution procedure to two parts: an offline part where a small number of judiciously-chosen solutions are precomputed via a greedy algorithm, and an online part in which the solution for any new parameter value is efficiently approximated by a projection onto the low-dimensional space spanned by the precomputed solutions.  This offline-online decomposition strategy is effective when one can afford an initial (offline) investment of significant computational resources so that certifiably accurate solutions for any other parameter value can be obtained in real-time (online).  

The classical RBM was originally developed for use with variational methods for approximating solutions to partial differential equations (PDEs) with affine dependence on the parameter.  The most popular of these variational approaches is the {\em Galerkin} method, derived by positing a solution ansatz in a subspace, and subsequently requiring that the projection of the PDE residual onto the same functional subspace is zero. An alternative approach for the solution of PDEs is to require that the PDE residual vanish at some predetermined collocation points.  These {\em collocation} methods are attractive because they are frequently easier to implement compared to Galerkin methods, particularly for time-dependent nonlinear problems \cite{TrefethenSpecBook,ShenTangBook,HesthavenGottlieb2007}.  In \cite{ChenGottlieb}, two of the authors developed a RBM suitable for collocation methods, and thus introduced a reduced collocation method (RCM). The RCM is extremely efficient and provides a reduced basis strategy for practitioners who prefer a collocation approach to solving PDEs, rather than a Galerkin approach.  Indeed, one of the two approaches in \cite{ChenGottlieb}, the empirical reduced collocation method, eliminates a potentially costly online procedure that is usually necessary for non-affine problems using a Galerkin approach.  The RCM's efficiency matches (or, for non-affine problems, exceeds) that of traditional Galerkin-based RBM approaches.

The RCM was developed for use with traditional collocation methods, particularly for pseudospectral methods \cite{HesthavenGottlieb2007}.  However, such collocation methods require a very structured grid which may be inconvenient when the spatial domain associated with the PDE has an irregular shape. When an irregular geometry is present, meshfree methods are a viable choice when compared to traditional meshed methods. Meshfree methods eliminate the need for finite-element-like meshes or adherence to symmetrical grid layouts, and have implementation costs that scale well with the spatial dimension; these properties are advantageous when compared to more standard mesh-based discretization approaches.

One particular mesh-free collocation method that we explore here is based on radial basis functions (RBFs). RBF methods have been widely used for scattered data interpolation and approximation in high dimensions \cite{Buhmann03,wendland,Fasshauer07}. RBF collocation methods for elliptic PDEs, based on global, non-polynomial interpolants, have been developed since the early 90s \cite{Kansa90,Kansa90b,KansaHon02,LarssonFornberg03}. RBF methods are collocation methods, implemented on scattered sets of collocation sites (commonly called \textit{centers}) and, unlike traditional pseudospectral methods, are not tied to a particular geometric structure. We employ a spatially-local variant of more traditional global RBF methods. RBF methods usually approximate differential operators with global stencils, but we employ a local approach which, inspired by its relation to finite-difference (FD) methods, are called RBF-FD methods \cite{tolstykh2003,shudingyeo2005,Wright2006,fornberglehto2011}. The RBF-FD method has the advantage of retaining high-order accuracy while improving computational efficiency by forming sparse operator matrices.

In this paper we extend the RBM strategy to include meshfree collocation methods, in particular RBF and RBF-FD methods.  We develop an algorithm that inherits the strengths of model-order reduction and geometric flexibility from RBM/RCM and RBF methods, respectively. This novel Reduced Radial Basis Function Method (\rrbm{}) is capable of achieving orders-of-magnitude speedup for solving parameterized problems on irregular domains. Our numerical results demonstrate exponential convergence with respect to the number of precomputed solutions. %for two- and three-dimensional test cases. 

The paper is organized as follows: In Section \ref{sec:rbm} we review both the RBM and RCM order-reduction algorithms. In Section \ref{sec:rbf} we present the particular radial basis function method we adopt (RBF-FD). The novel \rrbm{} algorithm is proposed in Section \ref{sec:r2bfm}, and we present numerical experiments that illustrate its performance for 2D- and 3D-problems. %Finally, we present our conclusions in Section \ref{sec:conclude}.

\section{The Least Squares Reduced Collocation Method}
\label{sec:rbm}
%\label{sec:alg}

In this section, we briefly review the reduced basis and least-squares reduced collocation methods originally introduced in \cite{ChenGottlieb}. RBMs aim to efficiently solve parameterized PDEs with certifiable error bounds. Repeatedly solving the full PDE for several values of the parameter is computationally onerous when the PDE itself is so complicated that a single solve is expensive. The RBM framework mitigates this cost first by carefully choosing a small set of parameter values at which the expensive PDE model is solved and stored. Once this expensive ``offline" procedure is completed, then the ``online" RBM algorithm computes the PDE solution at any new parameter value as a linear combination of the precomputed and stored solutions. The offline-online decomposition details of the algorithm ensure (1) that this procedure is accurate and (2) that the cost of the new parameter value solve is orders-of-magnitude smaller than a standard PDE solve.

The reduced basis method was invented in the late 1970s for nonlinear structural analysis \cite{Almroth_Stern_Brogan,Noor_Peters,Nagy}, and more broadly developed therein  \cite{Barrett_Reddien,Fink_Rheinboldt_1,Porsching,MatacheBabuskaSchwab2000,Peterson,Balmes}.  Recently, it has been systematically analyzed and applied to a wide variety of problems, see e.g.  \cite{Prudhomme_Rovas_Veroy_Maday_Patera_Turinici, Veroy_Prudhomme_Patera,CHMR_Sisc, Grepl_Patera, Haasdonk_Ohlberger, Nguyen_Rozza_Huynh_Patera, UrbanPatera2014, YanoPatera2013, YanoPateraUrban2014}, with \cite{Rozza_Huynh_Patera} containing extensive references. The RBM algorithm is traditionally applied to Galerkin discretizations of PDEs. However, recent work in \cite{ChenGottlieb, ChenGottliebMaday} develops a robust framework for applying the RBM algorithm to collocation discretizations of PDEs.  Since radial basis function discretization methods are collocative, we will later in the paper apply the RCM developed in \cite{ChenGottlieb}, in particular the Least Squares RCM.

To describe the RCM algorithm, we begin with a linear parametrized PDE of the form
\begin{align}\label{eq:pde}
  \mathbb{L} (\mu)\, u_\mu (\x) &= f(\x; \mu), & \x &\in \Omega \,\,\,{\subset \mathbb{R}^d}
\end{align}
with appropriate boundary conditions. Here, $\x$ is the spatial variable, $\mu$ is the parameter, $\mathbb{L}$ is the differential operator that depends on the parameter, $f$ is a forcing function, and $u$ is the unknown solution. The dimension of the spatial variable obeys $d \leq 3$ for most physical problems of interest, and so we will adopt the notation $\x = (x, y, z)^T$ for the components of $\x$. For any parameter $\mu= (\mu^1, \dots, \mu^p) \in \calD \subset \R^p$, a prescribed {$p$-}dimensional real parameter domain, we introduce a {\em discrete} differentiation operator $\mathbb{L}_\calN (\mu)$ approximating $\mathbb{L} (\mu)$ such that 
\begin{align}
\mathbb{L}_\calN (\mu)\, u^\calN_\mu (\x_j) = f(\x_j; \mu),  \label{eq:discpde}
\end{align}
is satisfied {\em exactly} on  a given set of $\calN$ collocation points $C^\calN = \{\x_j\}_{j=1}^\calN$. We assume that this discretization is such that the resulting approximate solution $u^\calN_\mu$ is highly accurate for \textit{all} parameter values $\mu \in \calD$, and refer to $u^\calN_\mu$  as the ``truth approximation''. With such a robust requirement on accuracy of the truth solution, $\calN$ will be sufficiently large so that solving \eqref{eq:discpde} is a relatively expensive operation that we wish to avoid performing too many times.

With this setup, the RCM algorithm in \cite{ChenGottlieb} proceeds first with an \textit{offline} stage where the algorithm is shown with a small number of expensive truth solves of \eqref{eq:discpde} along with some preprocessing of reduced operators, followed by an \textit{online} stage where computational savings are reaped whenever the algorithm is queried for the solution at a new parameter value. 

\subsection{The offline stage: choosing parameter values}\label{sec:rbm-offline-snapshots}
The first main goal of the offline stage is to choose $N$ parameter values $\mu_1, \ldots, \mu_N$ with $N \ll \calN$ such that the corresponding truth solution ensemble $u_{\mu_1}^\calN, u_{\mu_2}^\calN, \ldots, u_{\mu_N}^\calN$ has a span that accurately approximates $u^\calN_\mu$ for any $\mu \in \calD$. The truth solutions $u_{\mu_j}^\calN$ are frequently called ``snapshots". That it is even possible to generate such a collection of snapshots has been theoretically verified for several differential operators of interest \cite{Maday_Patera_Turinici_2,BuffaMadayPateraPrudhommeTurinici2011,BinevCohenDahmenDevorePetrovaWojtaszczyk}.  The algorithmic way in which this reduced set of parameters $\mu$ is chosen is via a greedy computation that successively chooses parameter values maximizing an error estimate.  The definition of this error estimate hinges on the formulation of a reduced approximation: For any $1 \leq n \leq N$, we seek the reduced solution $u^{(n)}_\mu$ defined as
\begin{align}\label{eq:rbm-surrogate}
  u^{(n)}_\mu(\cdot) = \sum_{j=1}^n c_j(\mu) u^\calN_{\mu_j}(\cdot), 
\end{align}
where the coefficients $\mathbf{c} = \left\{c_j(\mu)\right\}_{j=1}^n$ are computed by solving a least-squares residual problem on the truth solution nodes $\mathbf{\x} = C^\calN$:
\begin{subequations}\label{eq:rbm-approximation}
\begin{align}
  \label{eq:rbm-coefficients}
  \mathbf{c}(\mu) = \argmin_{\mathbf{d} \in \R^n} R_n(\mu; \mathbf{d}), \quad \mbox{with}\\
  \label{eq:rbm-offline-residual}
  R_n(\mu; \mathbf{d}) = \left\| f(\mathbf{\x}; \mu) - \sum_{j=1}^n d_j(\mu) \mathbb{L}_{\calN}(\mu) u^\calN_{\mu_j}(\mathbf{\x}) \right\|_{\ell^2}.
\end{align}
\end{subequations}
The ultimate goal of the error estimate is to approximate the error $u_\mu^{(n)} - u^\calN_\mu$ without directly forming the truth solution $u^\calN_\mu$. We denote this error estimate by $\Delta_n(\mu)$ and define it as:
\begin{align}\label{eq:deltan-definition}
  \Delta_n\left(\mu; \mathbf{c}(\mu)\right) =  \frac{R_n\left(\mu; \mathbf{c}(\mu)\right)}{{\sqrt{\beta_{LB}(\mu)}}}.
\end{align}
Above, $\beta_{LB}(\mu)$ is a lower bound for the smallest eigenvalue of $\mathbb{L}_\calN (\mu)^T \mathbb{L}_\calN (\mu)$, which effectively translates residuals into a bound for the actual errors. The accurate, computationally $\calN$-independent computation of $\beta_{LB}$ (and thus of $\Delta_n$) is in general one of the major difficulties in RBM algorithms, see e.g. \cite{HuynhSCM, CHMR-M2an, HKCHP}. These ingredients and an $\calN-$independent evaluation of $R_n(\mu; \mathbf{c})$ allow us finally to define how the parameter values are chosen through a greedy approach, shown in Algorithm \ref{alg:sketch}. If $\Delta_n$ can be computed in a $\calN$-independent fashion, then this greedy approach is computationally efficient.
\begin{algorithm}[h!]
  \caption{Outline of the offline RBM/RCM greedy algorithm}\label{alg:sketch}
  \begin{algorithmic}%[1]
\STATE {\bf 1.} $W_1 = \mbox{span}\left\{u^{\mathcal N}(\mu_1)\right\}$ (with $\mu_1$ arbitrarily chosen).
\STATE {\bf 2.} For {$i = 2,\dots, N$} do:
\begin{itemize}
\item [{\bf a).}] $\mu_i = \displaystyle
  \mbox{\rm arg}\hspace*{-1pt}\max_{\mu\in{\calD}}
 {\Delta_{i-1}(\mu)}$.  \quad {\bf b).} $W_i =
  \mbox{\rm span} \left\{ u_{\mathcal N}(\mu_j),~j\in\{1,\dots,i \} \right\}$.
\end{itemize}
  \end{algorithmic}
\end{algorithm}
Algorithm \ref{alg:sketch} presumes the ability to extremize $\Delta_n(\mu)$ over the continuous parameter domain $\calD$. In practice, one instead discretizes the parameter space, and scans it to generate the ``best" reduced solution space.  The first parameter value $\mu_1$ is randomly chosen, and the accurate truth solution $u^\calN_{\mu_1}$ is computed and stored, forming the first iterate of the solution space, $W_1$. Then for $n \geq 2$, we select the parameter value whose truth solution $u^{\calN}_\mu$ is worst approximated by the reduced solution $u^{(n)}_\mu$; this is accomplished with the error estimator $\Delta_n(\mu)$. The formulation of $\Delta_n(\mu)$ is rigorous, and so one can certify the maximum error over the parameter domain, stopping the iteration whenever this error reaches a desired tolerance level. In practice, a variant of the modified Gram-Schmidt transformation is applied to generate a more stable basis of $W_n$. An acceptable, certifiable error tolerance is usually reached with only $N \ll \calN$ truth snapshots.

\subsection{The offline stage: formation of reduced-order operators}\label{sec:rbm-offline-affine-stuff}
Assume that the snapshots $u^{\calN}_{\mu_{n}}$ for $n=1,\ldots, N$ are precomputed and stored from the previous section. The computation of the reduced-order solution $u^{(N)}_\mu$ and the residual $R_N(\mu; \mathbf{d})$ given by \eqref{eq:rbm-surrogate} and \eqref{eq:rbm-approximation} clearly requires $\mathcal{O}(\calN)$ operations as written because we must compute $\mathbb{L}_\calN(\mu) u^{\calN}_{\mu_j}$ for each new parameter value $\mu$. One condition that breaks this $\calN$-dependence is the assumption that the operator $\mathbb{L}(\mu)$ and the forcing function $f(\x;\mu)$ have \textit{affine} dependence on the parameter, i.e., that 
\begin{align}\label{eq:affine-assumption}
  \mathbb{L}(\mu) &= \sum_{q=1}^{Q_a} a_q^{\mathbb{L}}(\mu) \mathbb{L}_q, &
  f(\cdot;\mu) &= \sum_{q=1}^{Q_f} a_q^{f}(\mu) f_q(\cdot)
\end{align}
where the functions $a_q^{\mathbb{L}}(\mu)$ and $a_q^f(\mu)$ are scalar-valued and $\x$-independent, and the composite operators $\mathbb{L}_q$ and $f_q$ are $\mu$-independent. Many parameterized operators $\mathbb{L}(\mu)$ of interest do satisfy this assumption and there are effective strategies for approximating non-affine operators and functions by affine ones \cite{Barrault_Nguyen_Maday_Patera,Grepl_Maday_Nguyen_Patera}. We assume hereafter that the operator $\mathbb{L}(\mu)$ and the forcing function $f(\cdot;\mu)$ have affine dependence on $\mu$.

The affine dependence assumption allows us to precompute several quantities for use both later in the online stage, and in the offline process 
of selecting the $N$ snapshots. The discrete truth operator $\mathbb{L}_{\calN}$ and forcing function $f\left(\mathbf{\x}; \mu\right)$ likewise have an affine decomposition
\begin{align*}
  \mathbb{L}_{\calN}(\mu) &= \sum_{q=1}^{Q_a} a_q^{\mathbb{L}}(\mu) \mathbb{L}_{\calN,q}, &
  f(\mathbf{\x};\mu) &= \sum_{q=1}^{Q_f} a_q^{f}(\mu) f_q(\mathbf{\x}).
\end{align*}
Once the parameter values $\mu_1, \ldots, \mu_N$ are chosen, the following quantities may be computed and stored in the offline stage:
\begin{subequations}\label{eq:reduced-operators}
\begin{align}
  (M_{r,s})_{j,k} &\triangleq \left(\mathbb{L}_{\calN,r} u^{\calN}_{\mu_j}\right)^T \left(\mathbb{L}_{\calN,s} u^{\calN}_{\mu_k}\right), & 1 \leq j,k \leq N, &\;\;\; 1\leq r \leq s \leq Q_a, \\
  (g_q)_{j} &\triangleq \left(\mathbb{L}_{\calN,q} u^{\calN}_{\mu_j}\right)^T f(\mathbf{\x}; \mu_j), & j = 1,\ldots, N, &\;\;\; q= 1,\ldots, Q_f.
\end{align}
\end{subequations}
The resulting collection of matrices $\mathbf{M}_{r,s}$ are each $N \times N$, and the vectors $\mathbf{g}_q$ are $N \times 1$. 
%Below, we will make use of the equality $\mathbf{M}_{r,s} = \mathbf{M}_{s,r}^T$. 
Similar pre-computations are carried out to achieve $\calN$-independent evaluation of $R_n(\mu)$; see \cite{ChenGottlieb} for details.

\subsection{The online stage: computing a fast solution at a new parameter location}\label{sec:rbm-online}
All the operations that have $\mathcal{O}(\calN)$ complexity were completed in the previous offline sections. During the online stage, all operations are $\calN$-independent. Given a new parameter value $\mu^\ast$, we wish to compute the LSRCM approximation $u^{(N)}_{\mu^\ast}$ to the truth approximation $u^{\calN}_{\mu^\ast}$. This approximation is given by the coefficients $c_j$ in \eqref{eq:rbm-coefficients} with $n = N$. The formuluation \eqref{eq:rbm-coefficients} is a standard least-squares problem for the unknown coefficients. The coefficients $\mathbf{c}$ from \eqref{eq:rbm-coefficients} can be computed using the normal equations and \eqref{eq:reduced-operators}; we need only solve the following square, linear system:
\begin{gather}\label{eq:normal-equations}
  \mathbf{K}(\mu) \mathbf{c} = \mathbf{h}(\mu), \\\nonumber
  \begin{aligned} \mathbf{K}(\mu) &= \sum_{r,s=1}^{Q^a} a^{\mathbb{L}}_r\left(\mu^\ast\right) a^{\mathbb{L}}_s\left(\mu^\ast\right) \mathbf{M}_{r,s}, &\;\; 
  \mathbf{h}(\mu) &= \sum_{q=1}^{Q^f} a^{f}_q\left(\mu^\ast\right) \mathbf{g}_q\end{aligned}.
\end{gather}
The linear system \eqref{eq:normal-equations} is invertible and $N \times N$ so the coefficients $\mathbf{c}$ may be computed in an efficient and straightforward manner.  We emphasize that algorithmic methods to form and solve the above system have computational complexities that are \textit{independent} of the truth discretization parameter $\calN$. In this way, for each new $\mu^\ast$, we can use $\calN$-independent operations to compute the LSRCM approximation $u^{(N)}_{\mu^\ast}$. For several problems of interest \cite{ChenGottlieb}, the RCM solution $u^{(N)}_{\mu^\ast}$ converges spectrally to the truth solution $u^{\calN}_{\mu^\ast}$ with respect to $N$.

\section{The Local Radial Basis Function Method}
\label{sec:rbf}

The truth solution for the R$^2$BF method will be a large-stencil finite difference solution that is based on a meshfree radial basis function collocation method. Like finite difference methods, this RBF analogue has the advantages of low computational cost due to relatively sparse matrices. It also features flexibility in distributing collocation points for discretization on an irregular domain. The method is widely known as RBF in finite difference mode (RBF-FD) \cite{tolstykh2003,Wright2006}, or RBF differential quadrature (RBF-DQ) \cite{shudingyeo2005}.

The essence of this approach consists in defining differentiation matrices to transform a linear PDE operator $\mathbb{L}$ into a linear algebra problem $\mathbb{L}_{\calN}$. What differentiates one method from the other is the methodology of discretization. The local RBF discretization step consists of two ingredients: laying out points in the irregular domain and forming the differentiation matrices via high order local interpolants. In the following subsections we will describe our approach for each of these ingredients.  

\subsection{Background}

\subsubsection{Global approximation with radial basis functions}
\label{sec:global-rbf}
One of the simplest scenarios in which RBFs are employed is in global approximation methods. The basic idea is to form an interpolant that approximates a function $u(\x)$ based on data values $u_k$ given at scattered nodes $C^N = \left\{\x_k\right\}_{k=1}^\calN$. The RBF interpolant has the form
\begin{align}
  \label{eq:rbf-approximation}
  s(\x) =  \sum_{k=1}^{\calN} \lambda_k \phi(\varepsilon \|\x - \x_k\|),
\end{align}
where $\x$ denotes a point in $\R^d$, $\phi(r)$ is a radial basis function, $\varepsilon$ is a shape parameter, and $\| \cdot \|$ is Euclidean distance. For the moment, we assume that the nodes $\x_k$ (sometimes called \textit{centers} or \textit{grid points}) are a priori prescribed.

The function $\phi(\cdot)$ depends on the distance between points and not necessarily their orientation; one may generalize this approach to non-radial kernel functions but we stick to radial approximations in this paper. The choice of radial basis function $\phi(r)$ is clearly important since it directly affects the actual reconstruction $u(\x)$. Common choices of $\phi(r)$ are:
\begin{itemize}
\item Infinitely smooth functions: Multiquadrics (MQ) $\phi(r) = \sqrt{1 + r^2}$, Inverse Multiquadrics (IMQ) $\phi(r) = \frac{1}{\sqrt{1 + r^2}}$, and Gaussian (GA) $\phi(r) = e^{-r^2}$;
\item Piecewise smooth: Cubic $\phi(r) = r^3$, thin plate splines $\phi(r) = r^2 \ln(r)$;
\item Wendland's compactly supported piecewise polynomials \cite{Wendland1995}.
\end{itemize}
For a more comprehensive list,  we refer to \cite{Buhmann03,wendland,Fasshauer07}.  In this paper we use the inverse multiquadrics (IMQ). 

In order to solve for the coefficients $\bl = \left\{\lambda_j\right\}_{j=1}^\calN$, we collocate the interpolant $s(\x)$ to satisfy the interpolation conditions $s(\x_k) = u_k$, $k=1,\ldots,\calN$. This results in a system of linear equations
\begin{align}
\label{eq:rbf-system}
\begin{bmatrix}
\phi(\varepsilon \|\x_1 - \x_1\|) & \cdots & \phi(\varepsilon \|\x_1 - \x_\calN\|) \\
\vdots & \ddots & \vdots \\
\phi(\varepsilon \|\x_\calN - \x_1\|) & \cdots & \phi(\varepsilon \|\x_\calN - \x_\calN\|) \\
\end{bmatrix}
\begin{bmatrix}
\lambda_1 \\
\vdots \\
\lambda_\calN \\
\end{bmatrix}
=
\begin{bmatrix}
u_1 \\
\vdots \\
u_\calN
\end{bmatrix}
\Longrightarrow \bm{\Phi} \bl = \mathbf{u},
\end{align}
where the matrix $\bm{\Phi}$ is symmetric with entries $\Phi_{ij} = \phi(\varepsilon \|\x_i - \x_j\|)$ and $\mathbf{u} = [u_1, \ldots, u_\calN]^T$. The interpolation matrix $\bm{\Phi}$ is guaranteed to be non-singular for many choices of $\phi$ \cite{Micchelli86}.

Computing derivatives of the RBF interpolant $s(\x)$ in \eqref{eq:rbf-approximation} is a straightforward process, which can be simply done by summing the weighted derivatives of the basis functions. For example, with $x$ the first component of the vector $\x \in \R^d$, then
\begin{align}
 \label{eq:rbf-sx}
 \ppx{x}s(\x) =  \sum_{k=1}^{\calN} \lambda_k \ppx{x} \phi(\varepsilon \|\x - \x_k\|),
\end{align}
computes the first derivative of $s(\x)$ with respect to $x$ at any location $\x$. Since $\bl = \bm{\Phi}^{-1} \mathbf{u}$, then \eqref{eq:rbf-sx} can be written more compactly as
\begin{align}
\label{eq:rbf-diffmat}
\ppx{x} s(\x) = \bm{\Phi}_{x} \bm{\Phi}^{-1} \mathbf{u} \triangleq \bm{D}_{x} \mathbf{u},
\end{align}
where $\bm{\Phi}_{x} = \left[\ppx{x} \phi(\varepsilon \|\x - \x_1\|),\ldots,\ppx{x} \phi(\varepsilon \|\x - \x_\calN \|)\right]$. The matrix $\bm{D}_{x} = \bm{\Phi}_{x} \bm{\Phi}^{-1}$ has $\calN$ columns and is commonly called the RBF ``differentiation matrix". The number of rows in $\bm{D}_{x}$ depends on where the differentiated interpolant $\ppx{x} s$ should be evaluated, and we commonly want to evaluate on the same collocation points $C^\calN$. Thus, the matrix $\bm{\Phi}_x$ is $\calN \times \calN$ with entries
\begin{align}
\bm{\Phi}_x
=
\begin{bmatrix}
  \phi_{x_1}(\varepsilon \|\x_1 - \x_1\|) & \cdots & \phi_{x_1}(\varepsilon \|\x_1 - \x_\calN\|) \\
\vdots & \ddots & \vdots \\
\phi_{x_\calN}(\varepsilon \|\x_\calN - \x_1\|) & \cdots & \phi_{x_\calN}(\varepsilon \|\x_\calN - \x_\calN\|) \\
\end{bmatrix},
\end{align} 
where, in a slight abuse of notation, $\phi_{x_j}\left(\varepsilon \left\|\x_j - \x_k\right\|\right) = \ppx{x_j} \phi\left(\varepsilon \left\|\x_j - \x_k\right\|\right)$ is the partial derivative of the shape function with respect to the first component $x_j$ of $\x_j = (x_j, y_j, z_j)$. Higher order differentiation matrices or derivatives with respect to different variables (e.g $\bm{D}_y$,  $\bm{D}_{z}$,  $\bm{D}_{xx}$, etc) can be computed in the same manner.

Note that in general there are several theoretical and implementation aspects of global RBF approximation that are important to consider in practice. As an example, in cases where data values $u_k$ are sampled from smooth functions, the constructed interpolant $s(\x)$ can be highly accurate if infinitely smooth basis function $\phi(r)$ with small values of shape parameters $\varepsilon$ are used. However, as $\varepsilon$ becomes smaller, the basis functions becomes ``flatter" \cite{DriscolFornberg00r}, which leads to an ill-conditioned matrix $\bm{\Phi}$. Techniques to mitigate these kinds of interpolation instability issues include contour integration \cite{FornbergWright2004}, QR Decomposition \cite{FornbergPiretSisc2007,FornbergLarssonFlyer2009,FasshauerMcCourt2012}, Hilbert-Schmidt SVD methods \cite{cavoretto2014}, and SVD methods utilizing rational interpolants \cite{gonnet2011}. RBF methods that do not employ any of those techniques are usually called ``RBF-Direct". In this work, we only utilize RBF-Direct approximation methods to form the RBM truth approximation for accurately solving problems on irregular geometries. The use of more stable RBF interpolant algorithms is not the central focus of this work and will be left for future study.

The interpolation instability issues result not only from linear algebraic considerations. A judicious placement of nodes plays a crucial role through the classical problem of interpolation stability, as measured by Lebesgue constants and manifested through the Runge phenomenon. In one dimensional cases, oscillatory behavior near the boundaries do appear when equally-spaced points are used as $\calN$ becomes larger. This empirical observation is supported by potential-theoretic analysis \cite{Platte2005,PlatteIMA}. 
A stable approximation scheme for analytic functions on equally-spaced samples cannot converge exponentially \cite{PlatteTrefethenKuijlaars}. For higher dimensional cases, the precise distribution of scattered points is not yet well-understood although some numerical evidence is shown in \cite{FornbergLarssonFlyer2009}. 

For this paper and for modest size $\calN$, we use an algorithm that is based on the power function \cite{marchi_near-optimal_2005} to select optimal distribution of nodes. For this purpose, we will assume that $\phi$ is a positive-definite function, meaning that for any collection of $\calN$ distinct nodes $\x_k$, the interpolation matrix $\bm{\Phi}$ defined in \eqref{eq:rbf-system} satisfies
\begin{align}\label{eq:positive-definiteness}
  \mathbf{v}^T \bm{\Phi} \mathbf{v} &> 0, & \forall\; \mathbf{v} &\in \R^{\calN}
\end{align}
We make this assumption for two reasons: it ensures a unique solution to \eqref{eq:rbf-system}, and it allows us to construct a well-defined discrete norm on vectors. The IMQ basis functions that we use here satisfies the positive-definite condition.

\subsubsection{The native space norm}

Given a positive-definite function  $\phi$, we introduce the collection of functions $\phi^{x}$ centered at every location in the physical domain:
\begin{align}\label{eq:v-definition}
  \mathcal{V} &= \left\{ \phi^{\x} \left(\cdot\right) \, \big| \,\x \in \Omega\right\}, & \phi^{\x}(\y) &\triangleq \phi\left(\varepsilon\left\|\y - \x\right\|\right)
\end{align}
and define a proper norm on any function formed from a finite linear combination of elements in $\mathcal{V}$:
\begin{align}\label{eq:h-norm}
  u(\cdot) &= \sum_{k=1}^{\calN} \lambda_k \phi^{\x_k}(\cdot), & \left\| u \right\|^2_{\mathcal{H}} &\triangleq \sum_{1 \leq j,k \leq \calN} \lambda_j \lambda_k \phi^{\x_k}(\x_j).
\end{align}
Above, we have used $\x_k$ to indicate any selection of distinct points from $\Omega$. The \textit{native space} associated with the function $\phi$ is the $\|\cdot\|_\mathcal{H}$ closure of finite linear combinations of $\mathcal{V}$. The native space is a Hilbert space and we will denote it by $\mathcal{H}$. This construction is relatively abstract, but it is wholly defined by $\phi$, and one can characterize $\mathcal{H}$ as being equivalent to more standard $L^2$ Sobolev spaces by considering the decay rate of the Fourier transform of $\phi$ \cite{wendland}.

The inner product on $\mathcal{H}$ for $u = \sum_k \lambda_k \phi^{\x_k}$ and $v = \sum_k \rho_k \phi^{\x_k}$ is given by
\begin{align}\label{eq:h-inner-product}
	\left\langle u, v \right\rangle_{\mathcal{H}} \triangleq \sum_{1 \leq j,k \leq \calN} \lambda_j \rho_k \phi^{\x_k}(\x_j).
\end{align}
Given data $\bm{u}$ as in \eqref{eq:rbf-system}, we will use the notation $\|\bm{u}\|_{\mathcal{H}}$ to denote the corresponding $\mathcal{H}$-norm of the global interpolant \eqref{eq:rbf-approximation} and from \eqref{eq:h-norm} this norm is
\begin{align}\label{eq:discrete-h-norm}
  \left\| \bm{u} \right\|_{\mathcal{H}}^2 \triangleq \bm{\lambda}^T \bm{\Phi} \bm{\lambda} = \bm{u}^T \bm{\Phi}^{-1} \bm{u} = \left\| \bm{S}^{-1} \bm{u} \right\|^2,
\end{align}
with $\left\| \bm{v} \right\|$ the standard Euclidean norm on vectors $\bm{v}$, and $\bm{S}$ is any matrix satisfying $\bm{S} \bm{S}^T = \bm{\Phi}$. For concreteness, we will take $\bm{S}$ to be the positive-definite square root of $\bm{\Phi}$, i.e. since from \eqref{eq:positive-definiteness} $\bm{\Phi}$ is diagonalizable with positive spectrum: 
\begin{align*}
  \bm{\Phi} = \bm{V} \bm{\Lambda} \bm{V}^T \hskip 15pt \Longrightarrow \hskip 15pt
  \bm{S} = \bm{V} \sqrt{\bm{\Lambda}} \bm{V}^T
\end{align*}
However this choice is not necessary and in what follows one can replace $\mathbf{S}$ by, e.g., the Cholesky factor for $\bm{\Phi}$.

The norm on vectors $\|\cdot\|_{\mathcal{H}}$ defined above will be the RBF-analogue of a continuous norm in our RBM collocation framework.

\subsubsection{Choosing nodes: The Discrete Power Function method}\label{sec:rbf-power-function}

Traditional collocation methods, such as Fourier or Chebyshev pseudospectral methods, require a particular grid structure. Avoiding this restriction on irregular geometries is one of the major reasons why we turn to radial basis functions, which are mesh-free. However, it is well known that the accuracy of RBF methods is heavily influenced by the location of the centers $\x_k$. While some RBF nodal arrays are known to produce accurate reconstructions, these are mainly restricted to canonical domains -- tensor product or symmetric domains. 

\begin{wrapfigure}{r}{0.35\textwidth}
\includegraphics[width=.375\textwidth]{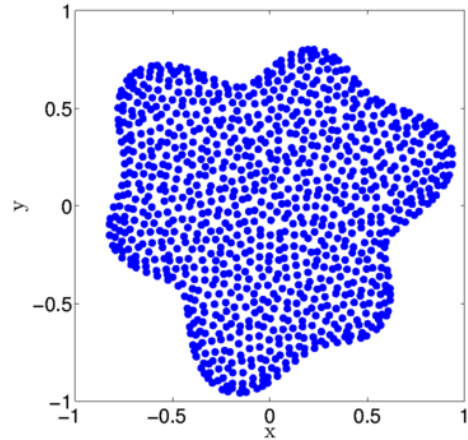}
\caption{\small RBF points resulting from the Power Function method}
\label{fig:pointsmatrix1}
\end{wrapfigure}

We are interested in computations on irregular domains, and therefore require a method for selecting region-specific nodes that will enhance the accuracy of the reconstruction. This is important since the inaccuracy of the RBF truth approximation will lead directly to that of the LSRCM 
solution. To generate nodes we make use of the Power Function Method applied on a discrete candidate set. We present a short discussion of this method in an effort to keep our presentation self-contained. Simplistic and effective, it relates to the reduced basis method by employing the same type of greedy algorithm. The interested reader may refer to \cite{marchi_near-optimal_2005} for a thorough discussion of this method; we provide an alternate description below in the context of the reduced basis framework.

Recall that $\Omega$ represents the physical domain, and that $\mathcal{V}$ from \eqref{eq:v-definition} is the collection of RBF shape functions centered at every point in $\Omega$. We consider this space as a collection of parameterized functions; the parameter is the nodal center $\x$.

The Power Function method selects RBF nodal centers by forming a reduced basis approximation to $\mathcal{V}$. From Algorithm \ref{alg:sketch}, we see that the reduced basis method greedily forms an approximation space by computing parametric values that maximize an error criterion. This error criterion is defined in terms of an error norm. For the space $\mathcal{V}$, it is natural to choose this norm to be $\|\cdot\|_{\mathcal{H}}$, the norm on the $\phi$-native space $\mathcal{H}$. Then applying an RBM offline selection of parameter values from $\mathcal{V}$ results in following optimization scheme, which mirrors Algorithm \ref{alg:sketch}:
\begin{align}\label{eq:rbm-x-iteration}
  \x_{n+1} &= \argmax_{\x \in \Omega}\, \mathrm{dist}_{\mathcal{H}}\left( \phi^{\x}, \mathcal{V}_n\right) = \argmax_{\x \in \Omega} \left\| \phi^{\x} - P_{\mathcal{V}_n} \phi^{\x} \right\|_{\mathcal{H}}, \\
  \mathcal{V}_{n+1} &= \mathrm{span} \left\{ \phi^{\x_1}, \ldots, \phi^{\x_{n+1}} \right\},
\end{align}
where $P_{\mathcal{V}_n}$ is the $\mathcal{H}$-orthogonal projector onto $\mathcal{V}_n$. To start the iteration, $\mathcal{V}_0 = \left\{0\right\}$ is the trivial subspace. Since the inner product of $\mathcal{H}$ is \eqref{eq:h-inner-product} and $\phi$ is a radial kernel, then
\begin{align}\label{eq:native-space-inner-product}
  \left\langle \phi^{\x}(\cdot), \phi^{\y}(\cdot) \right\rangle_{\mathcal{H}} = \phi^{\x}(\y) = \phi^{\y}(\x).
\end{align}
Thus, all the inner products and norms for the optimization can be computed simply by evaluating the shape function $\phi$. The optimization \eqref{eq:rbm-x-iteration} is the Power Function method of \cite{marchi_near-optimal_2005}.

We implement this method on finite candidate sets, e.g. substituting $\Omega$ with $Y^M = \left\{\y_{1}, \ldots, \y_{M}\right\} \subset \Omega$ with large $M$. From \eqref{eq:native-space-inner-product}, we need only form the interpolation matrix $\bm{\Psi}$ with entries $(\Psi)_{n,m} = \phi^{\y_n}(\y_m)$ for $m,n=1, \ldots, M$. Then the first $\calN$ points produced by the iteration \eqref{eq:rbm-x-iteration} can be computed by standard numerical linear algebra operations on $\bm{\Psi}$. Either the first $\calN$ iterations of a full-pivoting $L U$ decomposition on $\bm{\Psi}$, or the first $\calN$ iterations of a pivoted Choleksy decomposition produce the first $\calN$ points of the Power Function method. In practice, we use the Cholesky decomposition method: because we only need the pivoting indices, the required work can be completed in $\mathcal{O}(\calN^2 + M \calN)$ time with $\mathcal{O}(\calN^2 + M)$ storage, and we need only perform this operation once as part of setting up the truth approximation. An example of the result of this algorithm for a two-dimensional domain is given in Figure \ref{fig:pointsmatrix1}.

\subsection{RBF-FD method}\label{sec:rbf-fd}

\subsubsection{Local RBF Finite Difference Differentiation Matrices}

The differentiation matrices obtained by using global RBF interpolant based on infinitely smooth basis functions, such as IMQ, produce dense matrices. This is due to the fact that all $\calN$ nodes in the domain are used to generate the interpolant. The cost for inverting the dense matrix $\bm{\Phi}$, though done only once, is manageable for modest size $\calN$ but can be prohibitively expensive as it grows. One way to avoid dense matrix operations is to borrow ideas from finite-difference approximations, whose differentiation matrices are sparse. However, weights (i.e. entries of differentiation matrices) are difficult to compute when local stencils are scattered and differ in number and distributions. In order to mitigate these issues, local RBF interpolants and derivatives are used to compute stencil weights instead of using Taylor series, as in finite-difference methods. This is a straightforward approach for computing flexible finite-difference-like weights. As mentioned in the previous sections, this approach is known as a generalized finite-difference method or as RBF-FD.

We will illustrate the process of generating differentiation matrices in 2D; the generalization to higher dimensions is straightforward. The $\calN$ nodal points in $\Omega$ chosen by the Power Function method in Section \ref{sec:rbf-power-function} are denoted as $C^\calN =\{\x_1,\dots,\x_{\calN}\}$. Let $C_{(j)} = \{\x_{j_k}: k = 1, \dots, n_{\rm {loc}}^j \} \subset C^\calN$ be the (local) set of  neighboring points of $\x_j$ with $\x_{j_1} = \x_j$.  $C_{(j)}$ will form a stencil for $\x_j$. We call the point $\x_j$ the \emph{master} node of the set $C_{(j)}$. All other $n^j_{\textnormal{loc}}-1$ points in the set $C_{(j)}$ are \emph{slave} nodes.  As a simple example, Figure \ref{fig:domfigure} illustrates $\calN=101$ collocation points on a irregular domain $\Omega$ with a local stencil $C_{(1)}$ of $9$ slave nodes. While one can vary the size of the local stencil $n^j_{\textnormal{loc}}$ with respect to the master location $\x_j$, we often set $n^j_{\textnormal{loc}}$ to be the same in order to guarantee that all local interpolants provide approximately the same accuracy. An approach with different local stencil sizes is useful when there are local accuracy considerations (e.g., boundary layers).

\begin{figure}[h]
\input{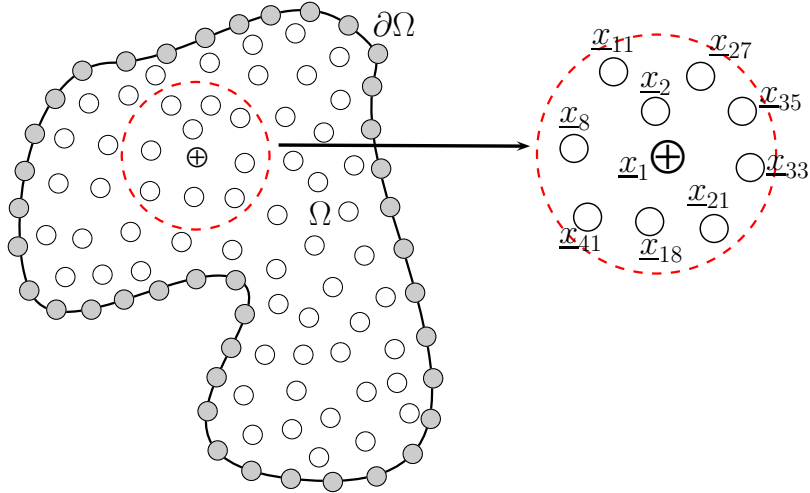}
\caption{\textbf{Left}: Collocation points on a 2D irregular domain. White nodes are points inside the domain and grey nodes are boundary points.  \textbf{Right:} An example of a 10-point-stencil for computing differentiation weights 
at $\x_1$. }
\label{fig:domfigure}
\end{figure}

Following the global formulation provided in Section \ref{sec:global-rbf}, the local interpolant $s_j(\x)$ with master node $\x_j$ takes the form
\begin{align}
  \label{eq:rbf-locapprox}
  s_{(j)}(\x) =  \sum_{k=1}^{n^j_\textnormal{loc}} \lambda_{j_k} \phi(\varepsilon \|\x - \x_{j_k}\|),
\end{align}
where  $\x_{j_k} \in C_{(j)}$. As in \eqref{eq:rbf-diffmat}, it follows that the differentiation matrix or derivative weights with respect to $x$ (the first component of $\x$) evaluated at the master node $\x_j$ can be easily obtained as
\begin{align}
  \bm{D}_{x(j)} = \left[\phi_{x_j}(\varepsilon \|\x_j - \x_{j_1}\|),\ldots,\phi_{x_j}(\varepsilon \|\x_j - \x_{j_{n^j_\textnormal{loc}}}\|)\right] \bm{\Phi}_{(j)}^{-1} = \bm{\Phi}_{x(j)} \bm{\Phi}_{(j)}^{-1},
\end{align} 
where ${D}_{x(j)}$ is of size $1 \times n^j_\textnormal{loc}$ and $\bm{\Phi}_{(j)}$ is the local interpolation matrix defined by the local problem \eqref{eq:rbf-locapprox}. Generating the $\calN \times \calN$ matrix $\bm{D}_1$ is done by computing ${D}_{x(j)}$ for $j = 1, \ldots, \calN$ and placing these vectors in the corresponding rows and columns of $\bm{D}_1$. The pseudocode for the process is shown in Algorithm \ref{alg:rbfdiffmat}.
\begin{algorithm}[htp]
  \caption{RBF-FD first derivative matrix with respect to $x$}
  \label{alg:rbfdiffmat}
  \begin{algorithmic}
  \STATE{\bf Input:} $C^\calN = \{\x_1,\ldots,\x_\calN \}$, $\phi(r)$, $\varepsilon$, $n_\textnormal{loc}$.
  \STATE{\bf Output:} $\bm{D}_x$
  \FOR{$j = 1$ to $\calN$}
  \STATE Find $C_{(j)} = \{\x_{j_k}: k = 1, \dots, n_{\textnormal{loc}} \} \subset C^\calN$, i.e $n_\textnormal{loc}$ nearest-neighbors of $\x_j$.
  \STATE Compute 
  ${D}_{x(j)} = \left[\phi_{x_j}(\varepsilon \|\x_j - \x_{j_1}\|),\ldots,\phi_{x_j}(\varepsilon \|\x_j - \x_{j_{n_\textnormal{loc}}}\|)\right] \bm{\Phi}_{(j)}^{-1} = {\bm{\Phi}_{x(j)}} \bm{\Phi}_{(j)}^{-1}$
  \STATE Store elements of  ${D}_{x(j)}$ as entries of $\bm{D}_x(j,j_1),\ldots,\bm{D}_x(j,j_{n_\textnormal{loc}})$ accordingly.
  \ENDFOR
  \end{algorithmic}
\end{algorithm}
Higher order differentiation matrices or derivatives with respect to different variables can be computed in the same manner
using the appropriate partial derivatives of $\phi(r)$ inside the for-loop of Algorithm \ref{alg:rbfdiffmat}. This algorithm generates $\calN \times \calN$ differentiation matrices with only 
$n_{\textnormal{loc}} \calN$ non-zero entries. The computational cost is $O(n_\textnormal{loc}^3 \calN)$ 
dominated by $\calN$ inversion processes of $\bm{\Phi}_{(j)}$. 
See \cite{Wright2006,bayona2011,bayona2012,fornberg2013,larsson2013} for stencil weights for RBF-FD for Gaussian and Multiquadric with constant shape parameters.

Once we have computed the RBF-FD differentiation matrices, we can use them to numerically solve boundary value problems or initial boundary value problems.
As a motivating example, we consider the following 2D boundary value problem
\begin{subequations}\label{eq:example-pde}
\begin{align}\label{eq:pde-operator}
  -u_{xx} - \mu^1 u_{yy} - \mu^2 u&= f(x,y), & (x,y) &\in \Omega,
\end{align}
with boundary conditions 
\begin{align}\label{eq:pde-boundary-conditions}
  u(x,y) &= g(x,y), & (x,y) &\in \partial \Omega.
\end{align}
\end{subequations}
The parameters $\mu^1$ and $\mu^2$ are constants. 

We then discretize $\Omega$ with $\calN_i$ nodes and $\partial \Omega$ with $\calN_b$ boundary nodes 
with a total
number of nodes $\calN = \calN_i + \calN_b$. 
We order the indices so that $\calN_i$ interior points are followed  
by $\calN_b$ boundary nodes. The discretized version of equation \eqref{eq:pde-operator} becomes
\begin{align}
  \label{eq:lineqL}
  - \bm{D}_{xx} \underline{u} - \mu^1 \bm{D}_{yy} \underline{u}   - \mu^2\bm{\mathcal{I}_i} \underline{u} =\underline{f}, 
\end{align}
or equivalently $\mathbf{L}(\underline{\mu}) \underline{u}  = \underline{f}$, with $\underline{\mu} = \left(\mu^1, \mu^2\right)$. The matrices $\bm{D}_{xx}$, $\bm{D}_{yy}$, which can be obtained from Algorithm \ref{alg:rbfdiffmat}, and consequently $\mathbf{L}(\underline{\mu})$ are of size $\calN_i \times \calN$. $\bm{\mathcal{I}_i}$ is an $\calN_i \times \calN$ matrix with values 1 at entries $(j,j)$ for $j= 1 \ldots \calN_i$ and zeros everywhere else.

For Dirichlet boundary conditions, the discretized version of the equation \eqref{eq:pde-boundary-conditions} becomes
\begin{align}
\label{eq:lineqLbc}
\bm{\mathcal{I}_b} \underline{u} = \underline{g},
\end{align}
where $\bm{\mathcal{I}_b}$ is an $\calN_b \times \calN_b$ matrix with values 1 at entries $(j,j)$ for $j=\calN_i+1,\ldots,\calN$ and zeros everywhere else. The systems \eqref{eq:lineqL}
and \eqref{eq:lineqLbc} are then augmented to form an $\calN \times \calN$ linear systems as shown in \eqref{eq:augsystem}.
%, which can be passed to a solver to compute $\underline{u}$.

\begin{align}
\label{eq:augsystem}
\begin{blockmatrixtabular}
\valignbox{
\begin{blockmatrixtabular}
\fblockmatrix       [0.8,1.0,0.8]{1.5in}{0.47in}{$\mathbf{L}(\underline{\mu})$} \\
\begin{blockmatrixtabular}
\fblockmatrix       [1.0,1.0,0.8]{1.0in}{0.47in}{$\bm{0}$} &
\Imatrix       [1.0,0.8,0.8]{0.45in}{0.47in}{$\bm{\mathcal{I}_b}$}
\end{blockmatrixtabular}
\end{blockmatrixtabular}
}
\hspace{-0.15in}
&
\valignbox{\fblockmatrix                    {0.15in}{1.0in}{$\underline{u}$}}&
\valignbox{\mblockmatrix                    {0.15in}{1.0in}{$=$}}&
\valignbox{
\begin{blockmatrixtabular}
\fblockmatrix       [0.8,0.8,1.0]{0.15in}{0.47in}{$\underline{f}$} \\
\fblockmatrix       [0.6,0.7,1.0]{0.15in}{0.47in}{$\underline{g}$}
\end{blockmatrixtabular}
}
\end{blockmatrixtabular}
\end{align}

\subsection{Numerical validation of the RBF-FD as truth solver}
\label{sec:poissontests}
Putting all the above pieces together, we have a radial basis function finite difference (RBF-FD) method. This section considers convergence studies to validate the accuracy of the method. 
To that end, we run two kinds of convergence tests on four equations \eqref{eq:awave2d} - \eqref{eq:diff3d}, emulating $h$-adaptive and $p$-adaptive refinement convergence studies from classical finite element approaches:
\begin{enumerate}
  \item ``$n_{\text{loc}}$ convergence" -- Refinement in the RBF-FD stencil size $n_{\text{loc}}$ while holding $\calN$ fixed is akin to $p$-refinement and so we expect exponential convergence in this case.
  \item ``$\calN$ convergence" -- Refinement in the truth parameter $\calN$ for a fixed RBF-FD stencil size $n_{\text{loc}}$ is akin to $h$-refinement and so we expect algebraic convergence as a result. 
\end{enumerate}
\begin{subequations}\label{eq:example-pdes-2d}
\begin{align}\label{eq:awave2d}
  \left\{
  \begin{aligned}
    -u_{xx} - \mu^1 u_{yy} - \mu^2 u &= f(\x), & \x &\in \Omega \\
    u &= g, & \x &\in \partial \Omega
  \end{aligned}
  \right.
  & \hskip 30pt
  \mu \in \calD = [0.1,4] \times [0,2]
\end{align}
\begin{align}\label{eq:diff2d}
  \left\{
  \begin{aligned}
    (1+\mu^1 x) u_{xx} + (1+\mu^2 y) u_{yy} &= f(\x), & \x &\in \Omega \\
    u &= g, & \x &\in \partial \Omega
  \end{aligned}
  \right.
  & \hskip 30pt
  \mu \in \calD = [-0.99,0.99]^2,
\end{align}
\end{subequations}
\begin{subequations}\label{eq:example-pdes-3d}
\begin{align}\label{eq:awave3d}
  \left\{
  \begin{aligned}
    -u_{xx} - \mu^1 u_{yy} - u_{zz}- \mu^2 u &= f(\x), & \x &\in \Omega \\
    u &= g, & \x &\in \partial \Omega
  \end{aligned}
  \right.
  & \hskip 30pt
  \mu \in \calD = [0.1,4] \times [0,2]
\end{align}
\begin{align}\label{eq:diff3d}
  \left\{
  \begin{aligned}
    (1+\mu^1 x) u_{xx} + (1+\mu^2 y) u_{yy} + z u_{zz}&= f(\x), & \x &\in \Omega \\
    u &= g, & \x &\in \partial \Omega
  \end{aligned}
  \right.
  & \hskip 30pt
  \mu \in \calD = [-0.99,0.99]^2.
\end{align}
\end{subequations}
These four equations are elliptic partial differential equations with a two-dimensional parameter $\mu = \left(\mu^1, \mu^2\right)$. 
Among them, equations \eqref{eq:example-pdes-2d} are spatially two-dimensional problems, and equations \eqref{eq:example-pdes-3d} are spatially three-dimensional problems. The problems \eqref{eq:awave2d} and \eqref{eq:awave3d} are examples of the Helmholtz equation, describing frequency-domain solutions to Maxwell's equations of electromagnetics. Examples \eqref{eq:diff2d} and \eqref{eq:diff3d} are steady-state diffusion equations with anisotropic diffusivity coefficients. 
Finally, we choose the inverse multiquadric (IMQ) shape function with shape parameter $\varepsilon=3$ for 2D 
and $0.75$ for 3D. 
To showcase the functionality of the method, we choose the computational domain $\Omega$ to be irregular: 
for the 2D problems, 
the domain $\Omega$ is centered at the origin with  boundary $\partial \Omega$ given by  the  parametric equation in polar coordinates $ r(\theta) = 0.8+0.1(\sin(6\theta)+\sin(3\theta)), \; \; 0 \leq \theta \leq 2 \pi $ (see Figure \ref{fig:pointsmatrix1}); for 3D ones, $\Omega$ is the closed interior of a solid object defined by the following parametric surface:
\begin{align}
\label{eq:parameq3D}
x^2 + y^2 + z^2 - \sin(2x)^2\sin(2y)^2\sin(2z)^2 = 1.
\end{align}

These same four equations will be used again to test the reduced solver later where we will have a 
fixed forcing function $f(x)$ and boundary data $g(x)$ (and thus parametric solution). For the purpose of 
validating the RBF-FD solver in this section, we 
choose  $f(x)$ and $g(x)$ such that the exact solution $u$ is known. 
They are listed in Table \ref{tab:exactsoln} and depicted in Figure \ref{fig:rbffdtestconv}.
\begin{table}[ht]
\renewcommand{\arraystretch}{1.8}
\begin{tabular}{|c|c|c|}
\hline
 & 2D & 3D\\
 \hline
 Test 1 & $u(x,y) = \sin(\pi x) \sin(\pi y)$ & $u(x,y,z) = \sin(\pi x) \sin(\pi y) \sin(\pi z)$\\
 \hline 
 Test 2 & $u(x,y) = e^{-20(x^2+y^2)}-x^2+y^3$ & $u(x,y,z) = e^{-20(x^2+y^2+z^2)}-x^2+y^3-z^2$\\
\hline
\end{tabular}
\caption{Parameter-independent exact solutions for the validation of the solver.}
\label{tab:exactsoln}
\end{table}
\begin{figure}[htb]
\begin{minipage}{\textwidth}
\begin{minipage}{0.24\textwidth}
\scalebox{0.475}{\includegraphics{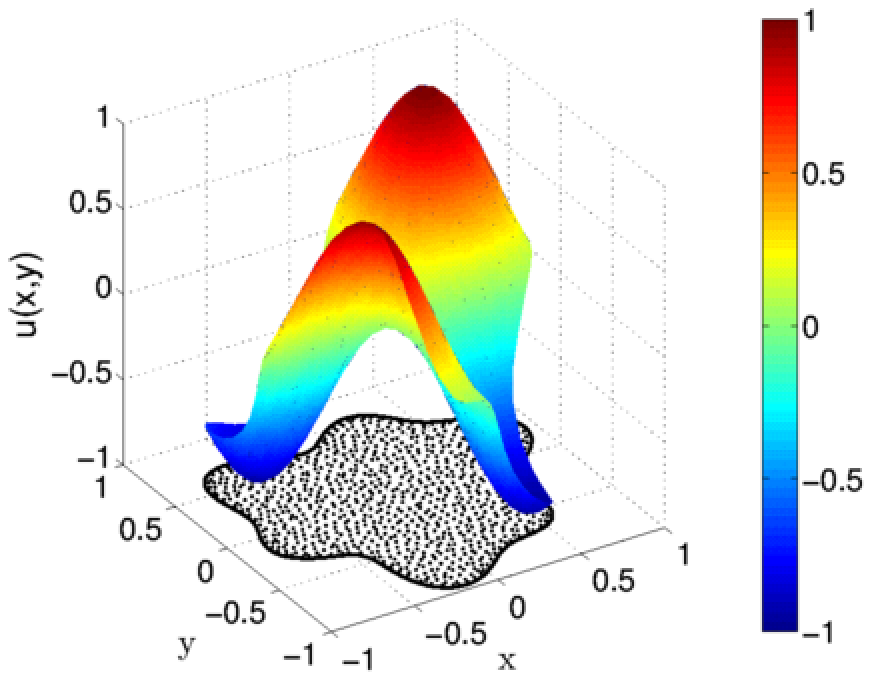}}
\end{minipage}
\begin{minipage}{0.24\textwidth}
\scalebox{0.475}{\includegraphics{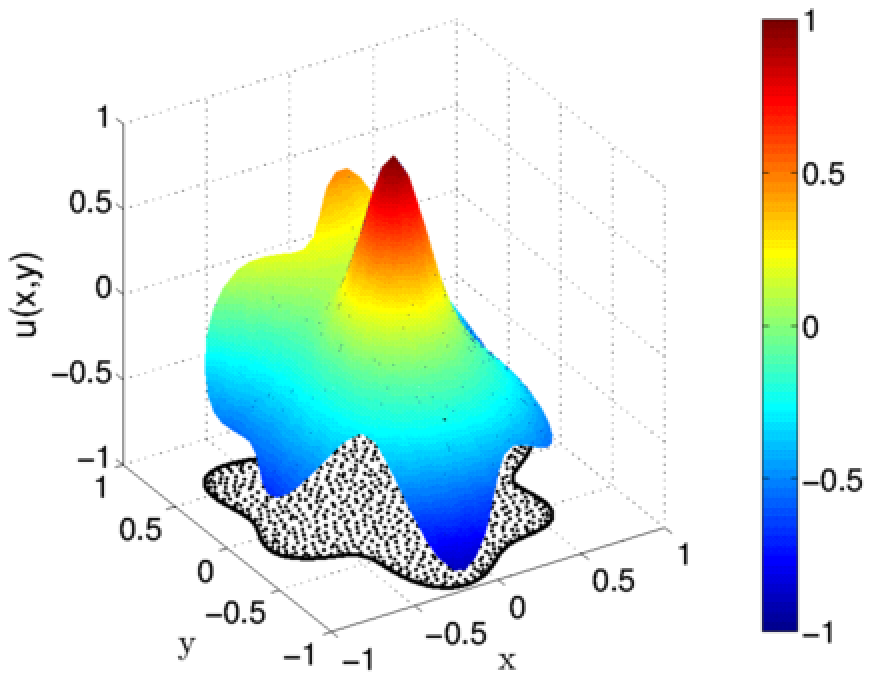}}
\end{minipage}
\begin{minipage}{0.235\textwidth}
\scalebox{0.55}{\includegraphics{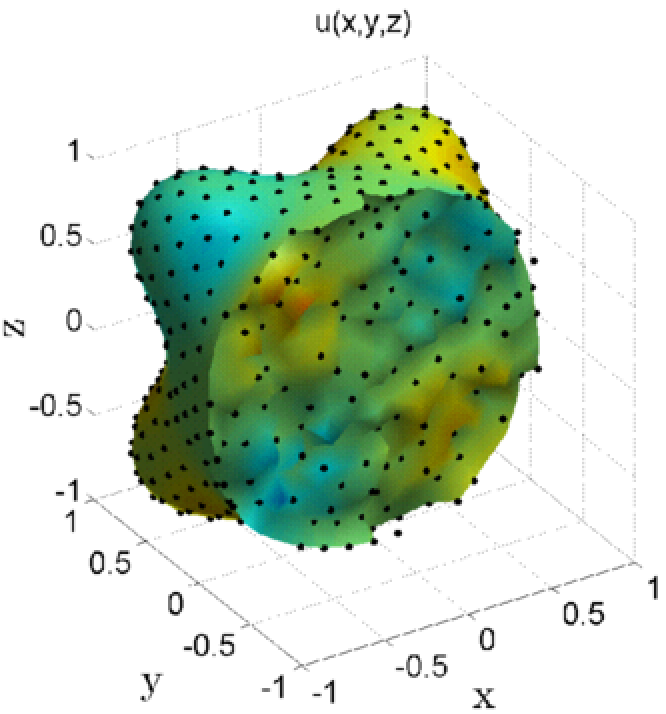}}
\end{minipage}
\begin{minipage}{0.235\textwidth}
\scalebox{0.55}{\includegraphics{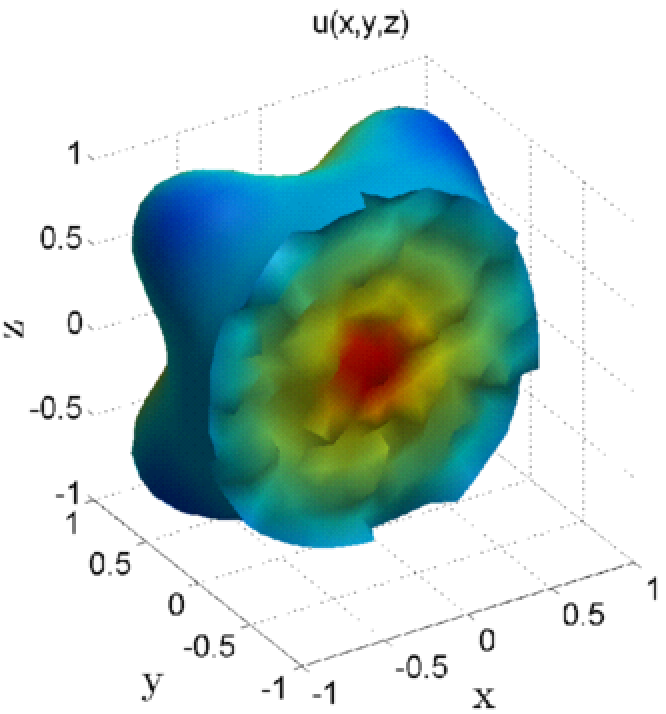}}
\end{minipage}
\end{minipage}
\caption{
Plot of the exact solutions: shown from left to right are 2D Test 1, 2D Test 2, 3D Test 1, 3D Test 2.}
\label{fig:rbffdtestconv}
\end{figure}

\subsubsection*{Accuracy of the truth solver}

The $\calN-$history of convergence results are shown in Figure \ref{fig:algcov_awdiff}, and that for $n_\text{loc}$ is shown in 
Figure  \ref{fig:geocov_awdiff}. The error shown in these figures is the $\mu$-maximum spatial $\ell^2$ norm over a candidate set $\Xi$ of $10,000$ equi-spaced parameters in $\calD$.

We observe that these numerical results indicate that the RBF-FD solver is robust and 
provides solutions that converge to the exact solution algebraically with respect to $\calN$, and exponentially with respect to $n_\text{loc}$. 
We note, however, that the error for the $n_\text{loc}$ convergence does level off  
with large number of local stencil points. 
This issue is a result of the ill conditioning observed in traditional RBF methods as the centers become closer.
It is partially mitigated by the use of greedily-selected optimal nodes obtained by using the discrete power 
function method described in Section \ref{sec:rbf-power-function}. Without it, the errors may eventually grow instead of 
flattening out due to ill-conditioning.

Once the {\it a priori} expectation of $hp-$type of convergence is confirmed, we have a reliable truth solver in RBF-FD. 
Moreover, these studies provide a reference for the accuracy of the truth approximations underlying the reduced solver in the next section. 
This accuracy will then provide a rough guideline for selecting the total number of reduced bases.

\begin{figure}[htb]
\begin{minipage}{\textwidth}
\begin{minipage}{0.24\textwidth}
\scalebox{0.75}{\includegraphics{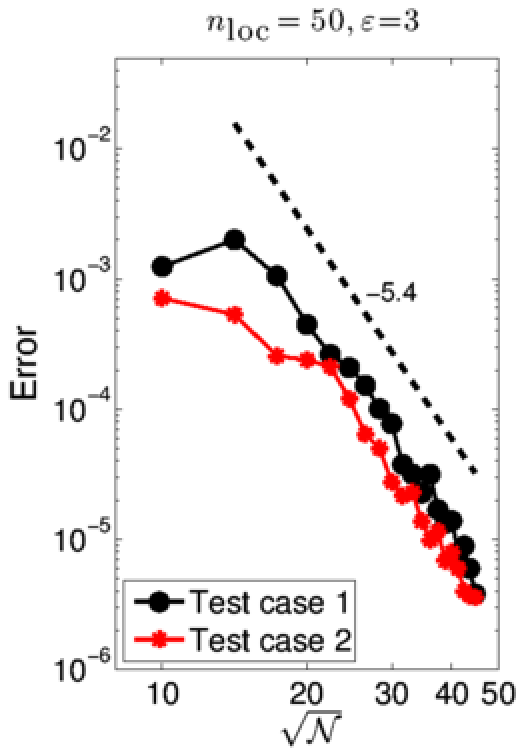}}
\center{(a)}
\end{minipage}
%\hspace{0.1in}
\begin{minipage}{0.24\textwidth}
\scalebox{0.75}{\includegraphics{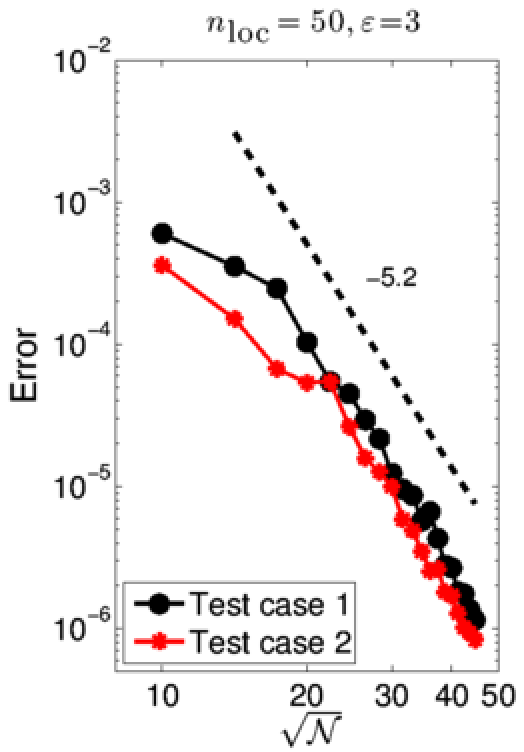}}
\center{(b)}
\end{minipage}
%\hspace{0.1in}
\begin{minipage}{0.24\textwidth}
\scalebox{0.75}{\includegraphics{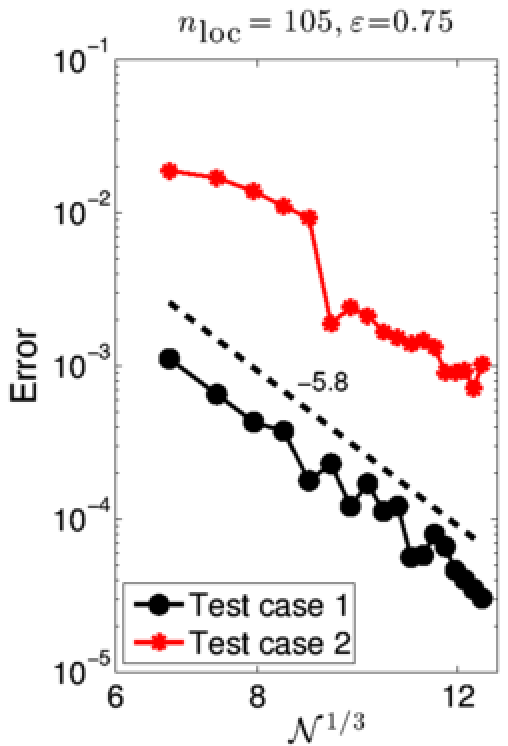}}
\center{(c)}
\end{minipage}
\begin{minipage}{0.24\textwidth}
\scalebox{0.75}{\includegraphics{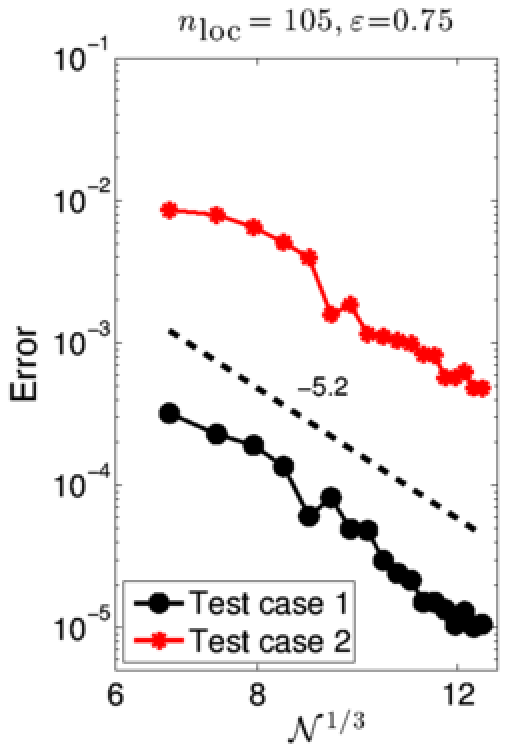}}
\center{(d)}
\end{minipage}
\end{minipage}
\caption{Convergence of the worst case error $\max_{\mu \in \Xi} \left\| u(\mu) - u^\calN(\mu)\right\|_{\ell^2}$ for 
equations \eqref{eq:awave2d} - \eqref{eq:diff3d} (left to right) as $\calN$ increases and $n_\text{loc}$ is fixed at $50$ for 2D and $105$ for 3D.}
\label{fig:algcov_awdiff}
\end{figure}
\begin{figure}[htb]
\begin{minipage}{\textwidth}
\begin{minipage}{0.24\textwidth}
\scalebox{0.75}{\includegraphics{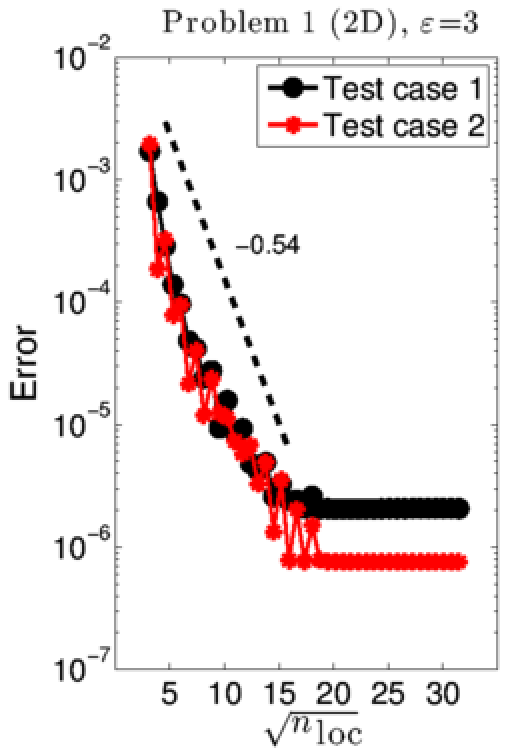}}
\center{(a)}
\end{minipage}
%\hspace{0.1in}
\begin{minipage}{0.24\textwidth}
\scalebox{0.75}{\includegraphics{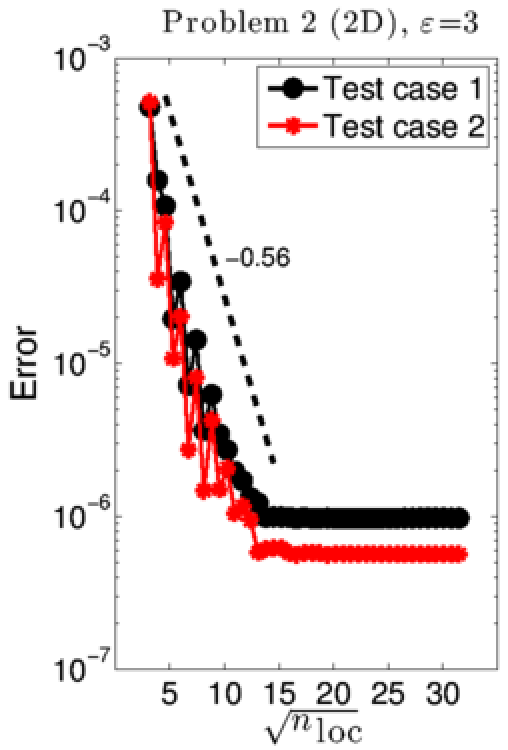}}
\center{(b)}
\end{minipage}
%\hspace{0.1in}
\begin{minipage}{0.24\textwidth}
\scalebox{0.75}{\includegraphics{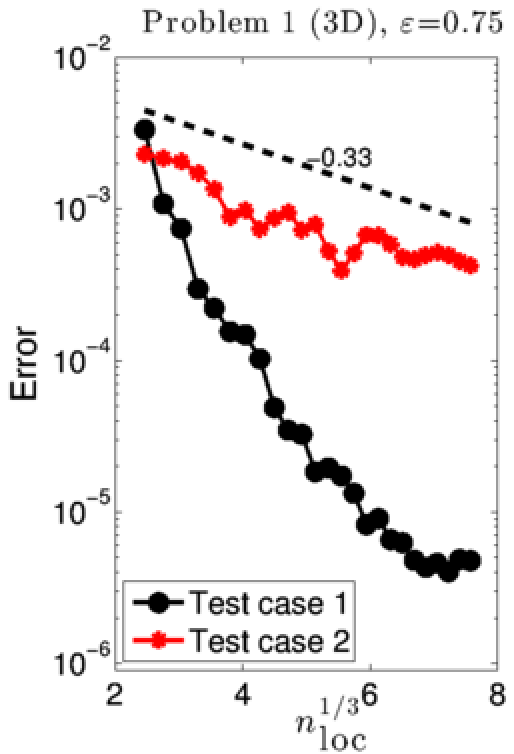}}
\center{(c)}
\end{minipage}
\begin{minipage}{0.24\textwidth}
\scalebox{0.75}{\includegraphics{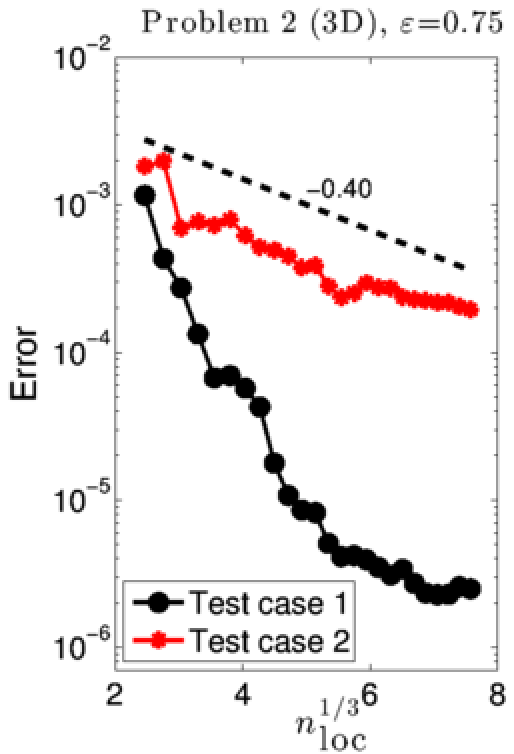}}
\center{(d)}
\end{minipage}
\end{minipage}
\caption{Convergence of the worst case error $\max_{\mu \in \Xi} \left\| u(\mu) - u^\calN(\mu)\right\|_{\ell^2}$ for 
equations \eqref{eq:awave2d} - \eqref{eq:diff3d} (left to right) as $n_\text{loc}$ increases and $\calN$ is fixed at $1000$ 
for 2D and $2046$ for 3D.}
\label{fig:geocov_awdiff}
\end{figure}

\section{The reduced radial basis function method}
\label{sec:r2bfm}

The novel contribution of this paper is the Reduced Radial Basis Function Method (\rrbm{}) that we describe in this section. 
The \rrbm{} algorithm uses the local RBF-FD method described in Section \ref{sec:rbf-fd} as the truth solver for a parameterized PDE of the general form \eqref{eq:pde}. We mainly consider irregular geometries, where this truth solver is advantageous compared to other solution methods. The tests run in Section \ref{sec:poissontests} show that this truth solver is highly accurate, featuring $h$-adaptivity in the parameter $\calN$, and $p$-adaptivity in the parameter $n_\text{loc}$. 
The \rrbm{} algorithm then uses the LSRCM  \cite{ChenGottlieb} to define the offline-online decomposition, defining the low-rank approximation space and the reduced-order operators.

\subsection{The \rrbm{} algorithm}

%\label{sec:precomp} 
In this section we outline the greedy algorithm for the (least squares) Reduced Radial Basis Function Method. Our truth approximation is given by the local RBF method outlined in Section \ref{sec:rbf}. The LSRCM method from Section \ref{sec:rbm} is naturally applicable to this solver because the local RBF method is a collocation method. However, we are left to specify the error estimate $\Delta_n$ given by \eqref{eq:deltan-definition}. To do this, we adapt the Chebyshev pseudospectral arguments from \cite{ChenGottlieb} to our local RBF case. 
In our RBF setting, we have a natural specification for the norm: the native space norm $\left\|\cdot\right\|_{\mathcal{H}}$ defined through $\bm{\Phi} = \bm{S} \bm{S}^T$.  To state our result, we define 
\begin{align*}
  \alpha_{UB}^S = \max_{\mathbf{w} \in \R^{\calN}} \frac{ \mathbf{w}^T \, \mathbf{S}^{-T}\, \mathbf{S}^{-1}\,  \mathbf{w} }{ \mathbf{w}^T \mathbf{w} },
\end{align*}
which relates the native space norm of a function to the standard Euclidian norm $\lVert \cdot \rVert$ of the corresponding vector by
\begin{equation}
  \lVert \bm{v} \rVert_{\mathcal{H}}^2 \leq \alpha_{UB}^S \lVert \bm{v} \rVert^2,
\label{eq:NnormToEnorm}
\end{equation}
where $\bm{v}$ is a vector of collocation evaluations of the truth approximation PDE solution. In this context, we can define two types of error estimators, one that works on residuals measured in the native space norm, and the second that is measured in the Euclidean norm:
\begin{subequations}\label{eq:error-estimator}
  \begin{align}\label{eq:error-estimator-native}
    \Delta_n^1(\mu) &\triangleq \frac{\sqrt{\alpha^S_{UB}} \lVert{\bm{S}^{-1}(\boldf^\calN - \mathbb{L}_\calN (\mu)u_\mu^{(n)})}\rVert}{{\sqrt{\beta_{LB}^S(\mu)}}}, \\\label{eq:error-estimator-euclidean}
  \Delta_n^2(\mu) &\triangleq \frac{\sqrt{\alpha^S_{UB}} \lVert{(\boldf^\calN - \mathbb{L}_\calN (\mu)u_\mu^{(n)})}\rVert}{{\sqrt{\beta_{LB}(\mu)}}} .
\end{align}
\end{subequations}
The $\beta$ factors that translate these residuals into (native-space-norm or Euclidean-norm) errors are given by
\begin{subequations}\label{eq:beta}
  \begin{align}\label{eq:beta-native}
    \beta_{LB}^S(\mu) &\triangleq \min_{\mathbf{w} \in \R^{\calN}} \frac{ \mathbf{w}^T \mathbb{L}_{\calN}^T(\mu)\, \mathbf{S}^{-T}\, \mathbf{S}^{-1}\, \mathbb{L}_{\calN}(\mu) \mathbf{w} }{ \mathbf{w}^T \mathbf{w} }, \\\label{eq:beta-euclidean}
    \beta_{LB}(\mu) &\triangleq \min_{\mathbf{w} \in \R^{\calN}} \frac{ \mathbf{w}^T \mathbb{L}_{\calN}^T(\mu)\, \mathbb{L}_{\calN}(\mu) \mathbf{w} }{ \mathbf{w}^T \mathbf{w} }.
  \end{align}
\end{subequations}
The estimators $\Delta_n^i$ for $i=1,2$ are computable, and they form rigorous error estimators in the native space.
\begin{thm}
  For any $\mu$, let $u_\mu^\calN(\mu)$ be the truth approximation solving \eqref{eq:discpde} and $u_\mu^{(n)}(\mu)$ be the reduced basis solution \eqref{eq:rbm-surrogate} solving \eqref{eq:rbm-approximation}.
Then we have $\lVert u_\mu^\calN - u_\mu^{(n)}\rVert_{\mathcal{H}} \le \Delta_n^i(\mu)$ for $i = 1, 2$.
\end{thm}
\begin{proof}
We have the following error equations on the $\calN$-dependent fine domain RBF grid thanks to the
equation satisfied by the truth approximation \eqref{eq:discpde}:
\[
\mathbb{L}_\calN (\mu) \left(u_\mu^\calN - u_\mu^{(n)}\right) = f - \mathbb{L}_\calN (\mu) u_\mu^{(n)},
\quad \bm{S}^{-1}\mathbb{L}_\calN (\mu) \left(u_\mu^\calN - u_\mu^{(n)}\right) = \bm{S}^{-1}(f - \mathbb{L}_\calN (\mu) u_\mu^{(n)}).
\]
%Multiplying both sides by the inverse of $\mathbb{L}_\calN (\mu)^T \mathbb{L}_\calN (\mu)$ and
Taking the the standard Euclidian norm and using basic properties of eigenvalues gives
\[ 
\lVert u_\mu^\calN - u_\mu^{(n)}\rVert \le \frac{\lVert f - \mathbb{L}_\calN
(\mu) u_\mu^{(n)}\rVert}{{\sqrt{\beta_{LB}(\mu)}}}, 
\quad \lVert u_\mu^\calN - u_\mu^{(n)}\rVert \le \frac{\lVert \bm{S}^{-1}(f - \mathbb{L}_\calN
(\mu) u_\mu^{(n)})\rVert}{{\sqrt{\beta_{LB}^S(\mu)}}}
\]
respectively. We then apply the inequality \eqref{eq:NnormToEnorm} to finish the proof.
\end{proof}

These error bounds can then be used to perform the offline LSRCM computations in Section \ref{sec:rbm-offline-snapshots}: determination of the parameter values $\mu^1, \ldots, \mu^N$ and the subsequent snapshots $u_{\mu^1}, \ldots, u_{\mu^N}$. In our experiments below, we have used the second estimate $\Delta_n^2$. The remaining LSRCM steps in Sections \ref{sec:rbm-offline-affine-stuff} and \ref{sec:rbm-online} are as described in those sections. We outline the entire offline algorithm in Algorithm \ref{alg:LSgreedy}.

\begin{remark}
  The estimates above work in the ``global" native space defined by the shape function $\phi$ and width parameter $\varepsilon$. However, our RBF approximation is a \textit{local} approximation and so it is perhaps more appropriate to use a ``local" native space norm in order to compute these estimates. However, our tests have shown that this distinction does not affect the result much in practice, although it does make a difference for very small local stencil sizes (e.g., $n_{\text{loc}}^{1/d} \leq 3$). In order to keep the method simple, we have therefore used the global native space norm as presented above for the RCM error estimate.
\end{remark}

\begin{algorithm}[htp]
  \caption{Least Squares R$^2$BFM{: Offline Procedure}}\label{alg:LSgreedy}
  \begin{algorithmic}%[1]
\STATE {\bf 1.} Discretize the parameter domain $\calD$ by $\Xi$, and denote the center of $\calD$ by $\mu_c$.
\medskip
\STATE {\bf 2.} Randomly select $\mu_1$ and solve the RBF problem $\mathbb{L}_\calN (\mu_1)\, u^\calN_{\mu_1}
(\x) = f(\x; \mu_1)$ for $\x \in C^\calN$.
\medskip
\STATE {\bf 3.} For $n = 2, \dots, N$ do%\IF{sth} \STATE do this; \ENDIF
\begin{itemize}
\item [{\bf 1).}] Form $\mathbb{A}_{n-1} = \left(\mathbb{L}_\calN\, u_{\mu_1}^\calN, \mathbb{L}_\calN\, u_{\mu_2}^\calN, \, \dots, \,
\mathbb{L}_\calN\, u_{\mu_{n-1}}^\calN\right)$.
\item [{\bf 2).}] For all $\mu \in \Xi$, solve
$\mathbb{A}_{n-1}^T\,\mathbb{A}_{n-1} \,\vec{c} = \mathbb{A}_{n-1}^T\,\boldf^\calN$ to obtain $u^{(n-1)}_\mu
= \sum_{j = 1}^{n - 1} c_j u^\calN_{\mu_j}$.
\item [{\bf 3).}] For all $\mu \in \Xi$, calculate $\Delta_{n-1}^2(\mu)$. Then set $\mu^n = argmax_{\mu}\,\,\Delta_{n-1}(\mu)$.
\item [{\bf 4).}] Solve the RBF problem $\mathbb{L}_\calN (\mu_n)\, u^\calN_{\mu_n}(\x) = f(\x; \mu_n)$ for $\x \in C^\calN$.
\item [{\bf 5).}]Apply a modified Gram-Schmidt transformation, with inner product defined by $(u,v) \equiv \left(\mathbb{L}_\calN (\mu_c)u, \mathbb{L}_\calN (\mu_c) v\right)_{L^2(\Omega)}$, on the basis $\left\{u^\calN_{\mu_1}, u^\calN_{\mu_2}, \dots, u^\calN_{\mu_n}\right\}$ to obtain a more stable basis $\left\{\xi_1^\calN, \xi_2^\calN, \dots, \xi_n^\calN\right\}$ to use as terms in the expansion \eqref{eq:rbm-surrogate} for the least squares reduced collocation method.
\end{itemize}
%\medskip
  \end{algorithmic}
\end{algorithm}

% Numerical Results

\subsection{Numerical Results}\label{sec:r2bfm_num}

We test our reduced solver on the four equations \eqref{eq:awave2d} - \eqref{eq:diff3d} listed in Section \ref{sec:poissontests}. 
In all cases, we employ the error estimator $\Delta^2_n$ from \eqref{eq:error-estimator-euclidean} to select snapshots and construct the reduced approximation space. 

\subsubsection{\rrbm{}: Two dimensional cases}

In this section, we test our reduced solver R$^2$BFM, taking $f = -10\sin(8x(y-1))$ for \eqref{eq:awave2d} and $f = e^{4xy}$ for \eqref{eq:diff2d}, 
both with homogeneous boundary conditions.
We discretize the 2D domain $\Omega$ with $\calN=1000$ RBF nodes with stencil of size $n_{\text{loc}} = 50$.  These $\calN$ RBF nodes (see Figure \ref{fig:pointsmatrix1}) are  selected by a greedy Cholesky algorithm out of $2984$ candidate points that are uniformly distributed on $\Omega$.  In the following numerical experiments we use inverse multiquadric RBFs with shape parameter $\varepsilon=3$. 

\begin{figure}[htb]
\includegraphics[width=0.45\textwidth]{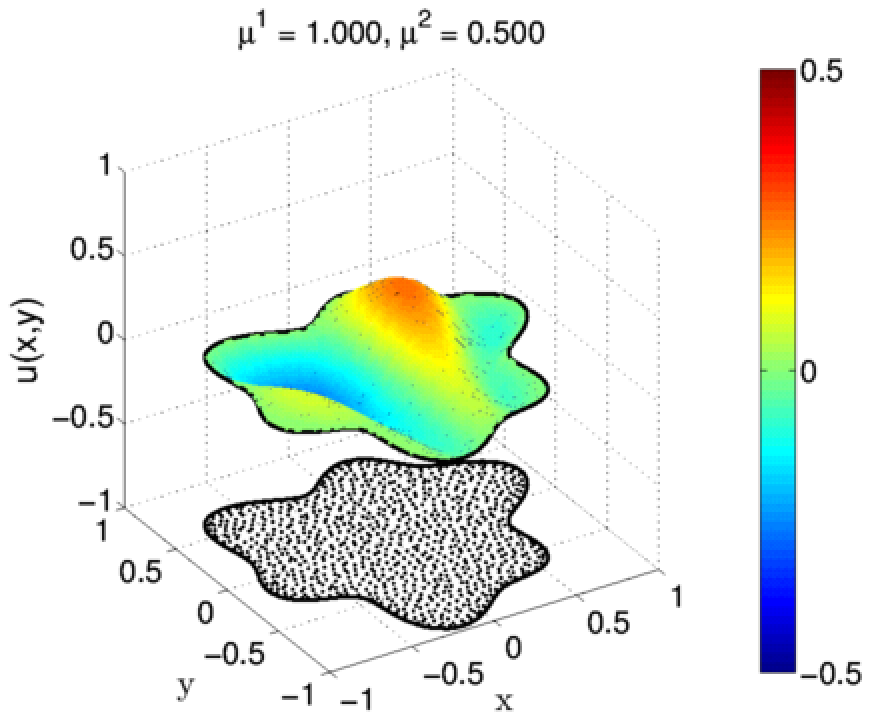}
\includegraphics[width=0.45\textwidth]{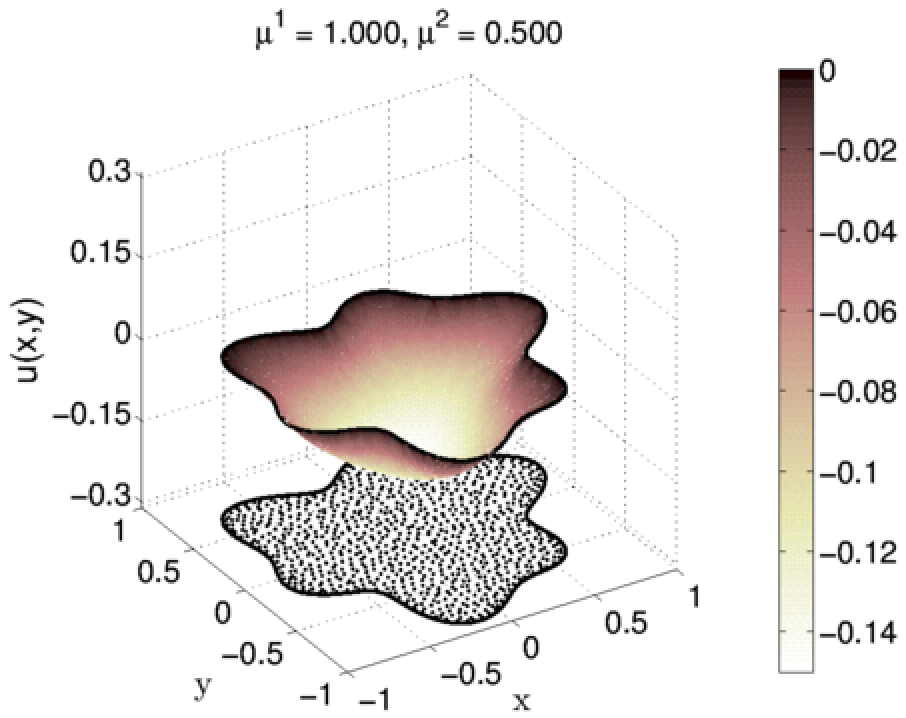}
\caption{Surface plots of the sample solutions for the 2D test problems \eqref{eq:awave2d} (left)  and \eqref{eq:diff2d} (right) at $\mu^1=1$ and $\mu^2=0.5$.
} \label{fig:solutions}
\end{figure} 
\begin{figure}[htb]
\includegraphics[width=0.5\textwidth]{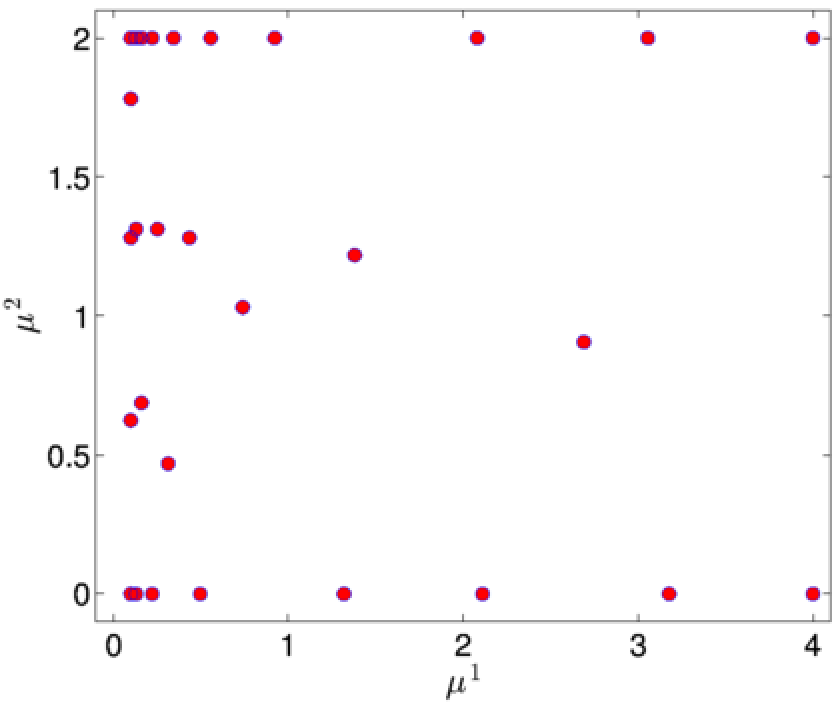}
\includegraphics[width=0.28\textwidth]{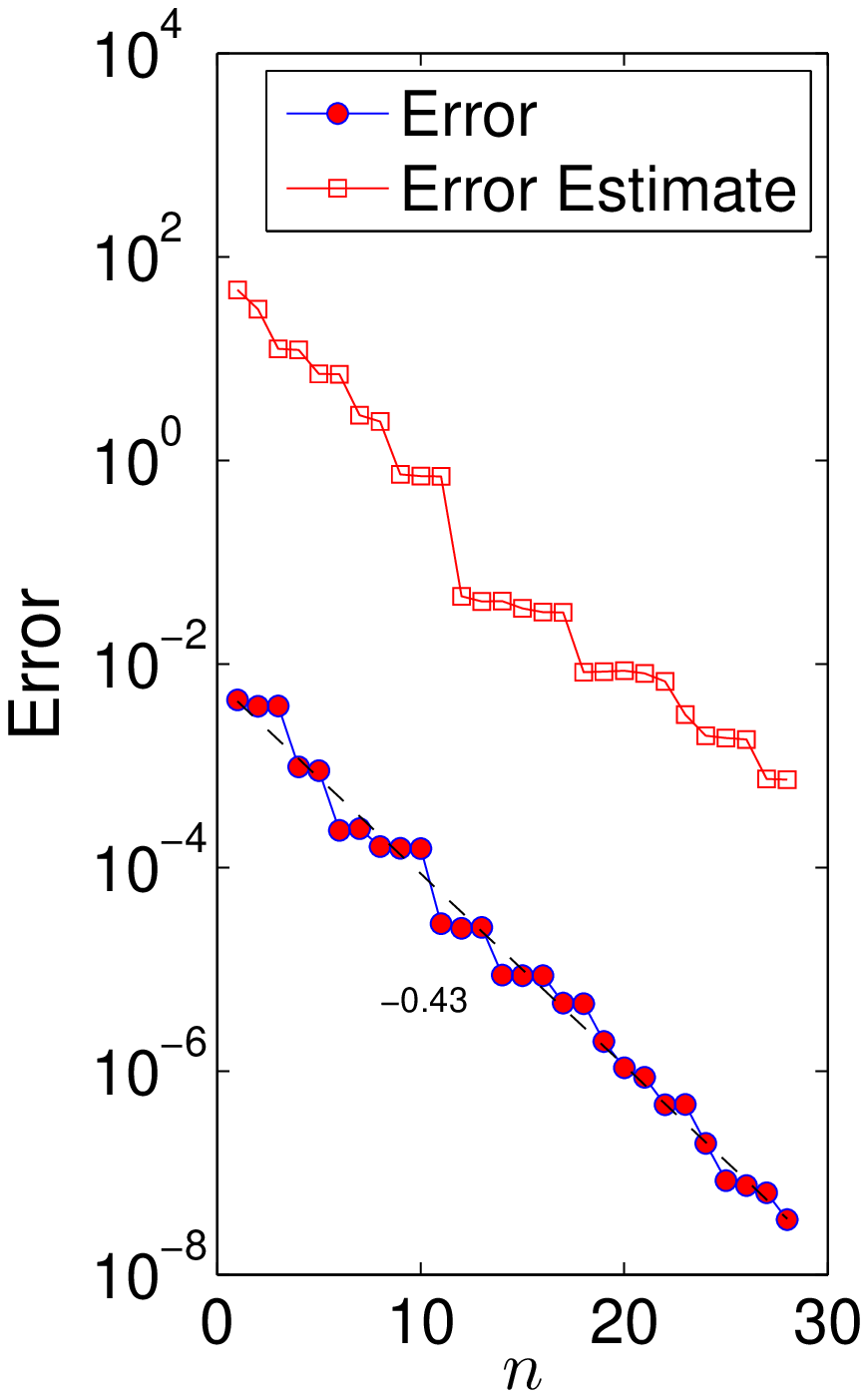}\\
\includegraphics[width=0.5\textwidth]{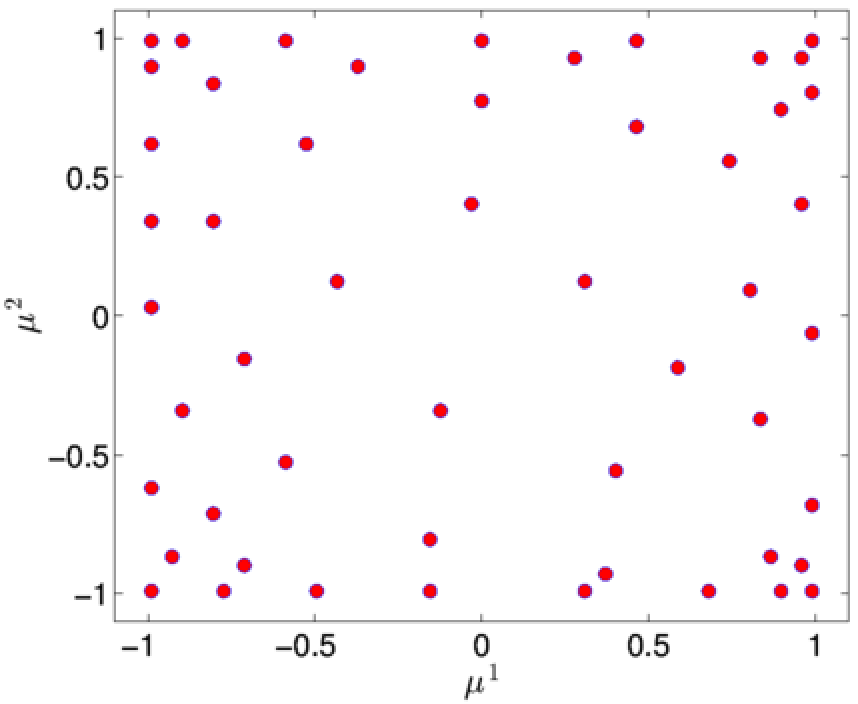}
\includegraphics[width=0.28\textwidth]{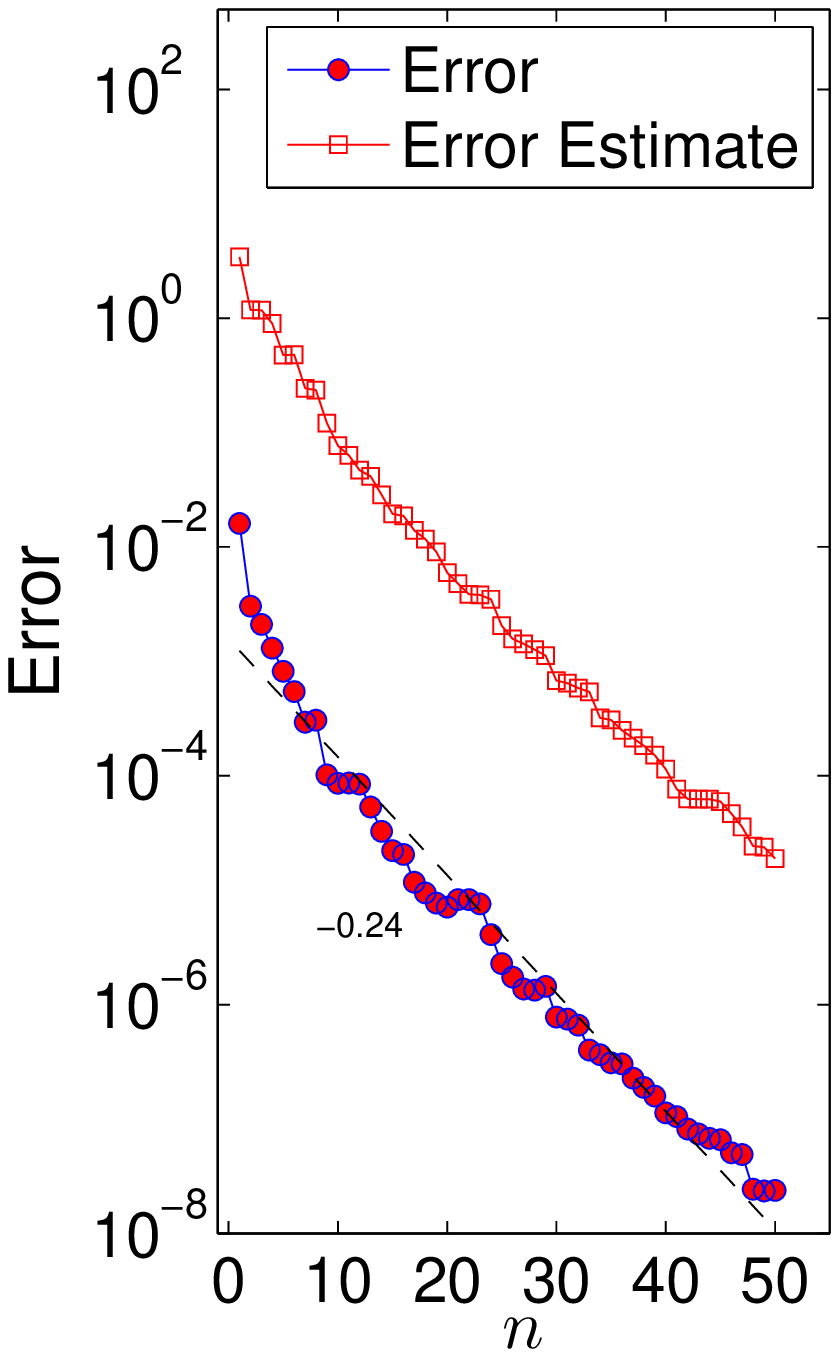}
\caption{Selected parameter values to build the reduced solution spaces (left), and the history of convergence of the RB approximations (right).  
Shown on top is for 2D problem \eqref{eq:awave2d}, and at the bottom is for \eqref{eq:diff2d}. 
}
\label{fig:RBsolutions}
\end{figure}

In Figure \ref{fig:solutions} we show the surface plots of the solutions to 
\eqref{eq:awave2d}  and \eqref{eq:diff2d}
at a particular parameter value $\mu^1=1$ and $\mu^2=0.5$. These images show the irregular domain shape and the complexity of the solution profile resulting from it. 
In Figure \ref{fig:RBsolutions} we see that the R$^2$BFM solution is converging exponentially to the RBF solution. Compared to the full RBF truth simulation with $\calN= 1000$ centers and local stencils of size $50$, the \rrbm{} algorithm achieves comparable accuracy with much less than $n = 50$ bases. We leave the details of the comparison for Section 
\ref{sec:comptime}.

\subsubsection{\rrbm{}: Three dimensional cases}
In this section we test our reduced solver R$^2$BFM on three-dimensional problems by taking $f = -10\sin(8x(y-1)z)$ for \eqref{eq:awave3d} and $f = e^{4xyz}$ for \eqref{eq:diff3d} with homogeneous Dirichlet boundary conditions. 
Truth approximation simulations were carried out using the inverse multiquadric RBF with $\varepsilon=0.75$, and local stencils of size $n_{\textnormal{loc}}=125$ on a mesh comprised of a total of $2046$ points from the Power Function method. Figure \ref{fig:aniswave3D} shows the solutions corresponding to $\mu^1=1$ and $\mu^2 = 0.5$.
\begin{figure}[htb]
\includegraphics[width=0.45\textwidth]{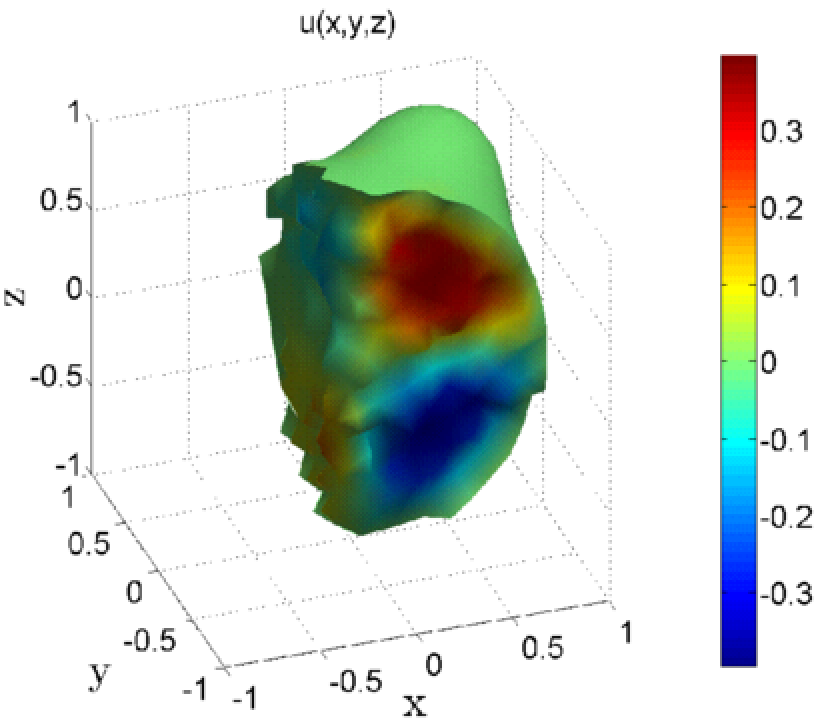}
\includegraphics[width=0.45\textwidth]{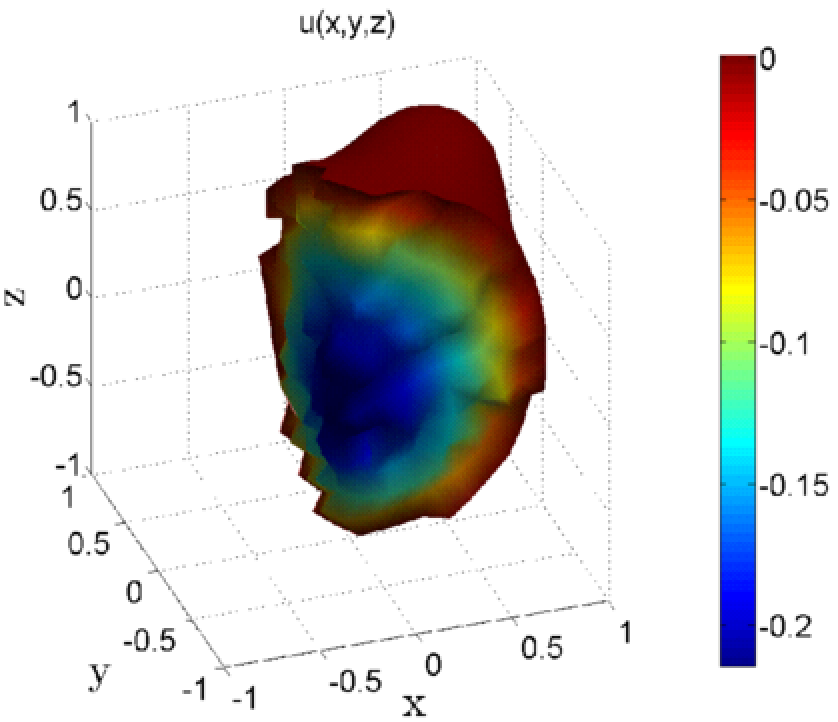}
\caption{Sample solutions at $\mu^1=1$ and $\mu^2=0.5$ for the 
3D problems \eqref{eq:awave3d} (left) and \eqref{eq:diff3d} (right).
}
\label{fig:aniswave3D}
\end{figure} 
Again, the reduced basis method converges exponentially (Figure \ref{fig:3DconvRB}) and provides accurate surrogate solutions 
with a very small number of basis functions. See Section 
\ref{sec:comptime} for the details of the comparison.
\begin{figure}[htb]
\includegraphics[width=0.45\textwidth]{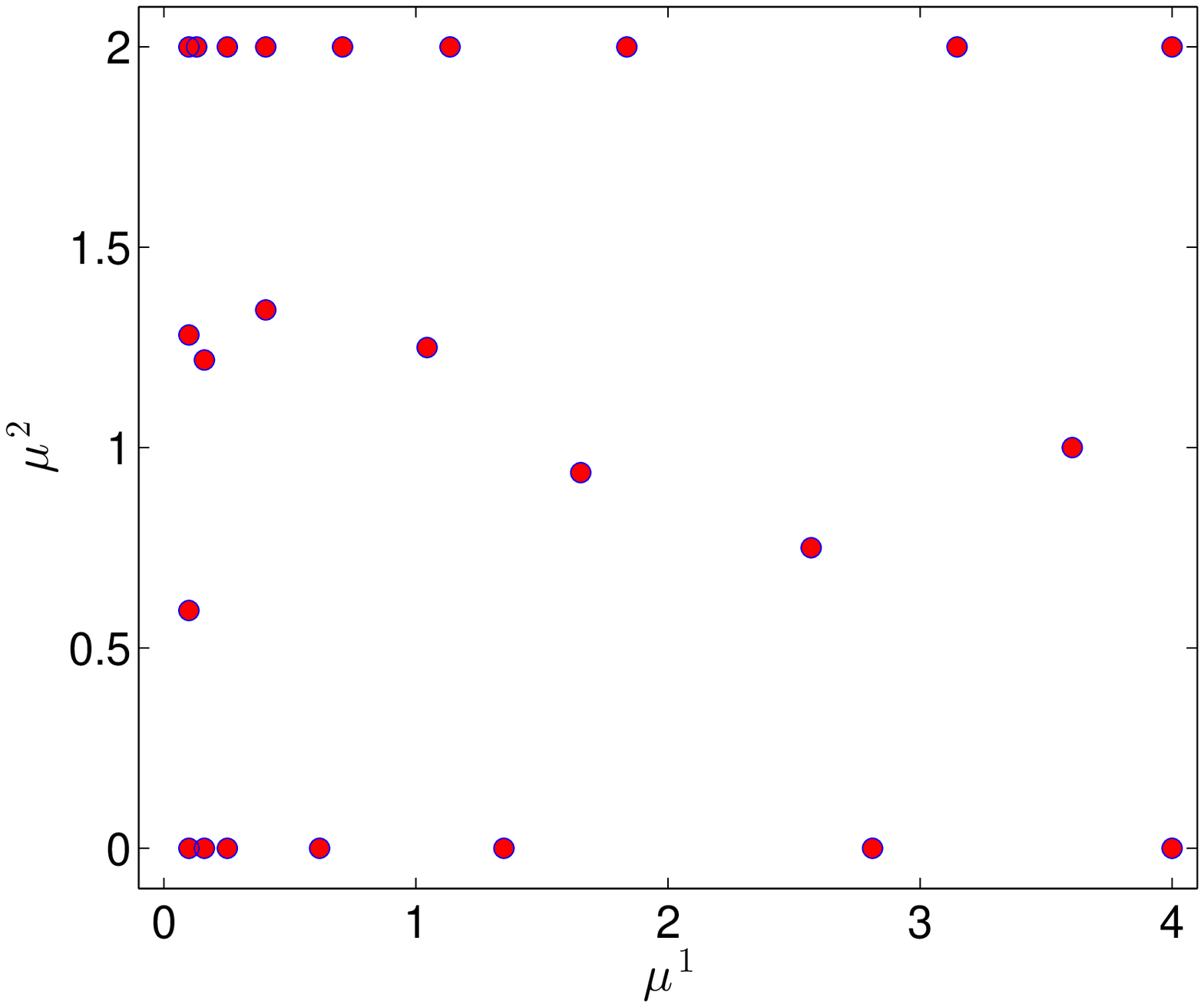}
\includegraphics[width=0.25\textwidth]{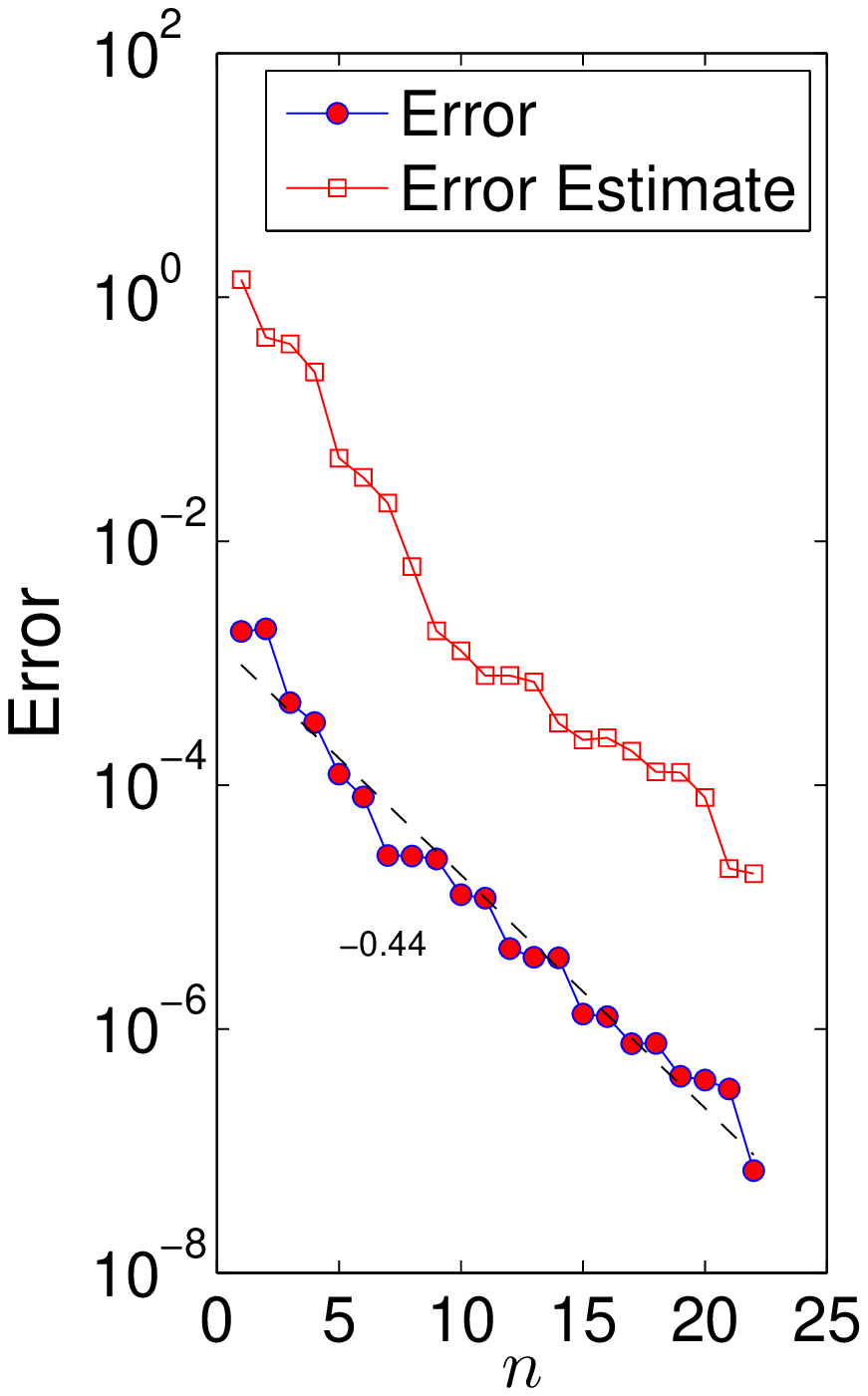} \\
\includegraphics[width=0.45\textwidth]{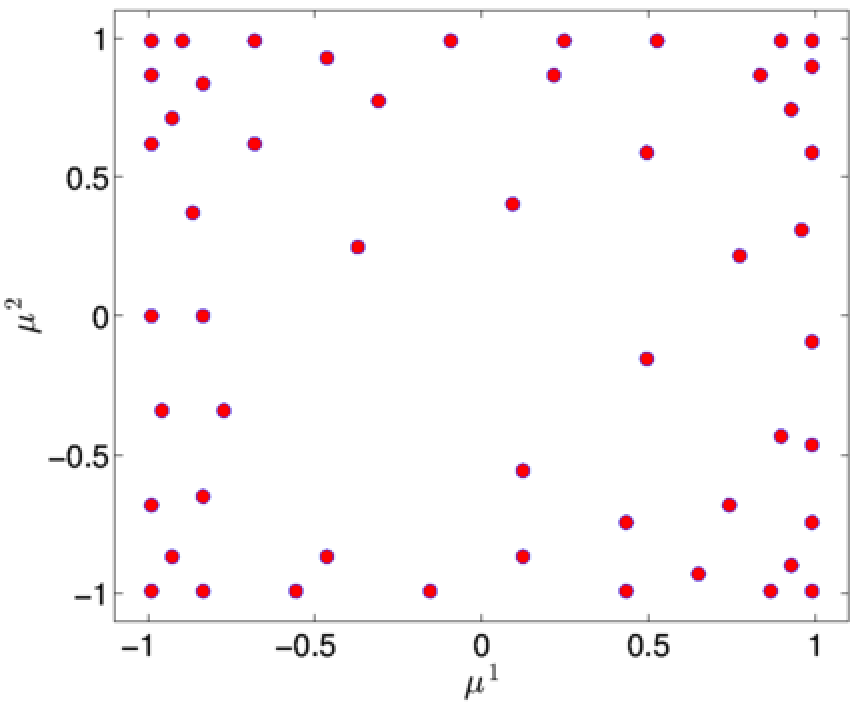}
\includegraphics[width=0.25\textwidth]{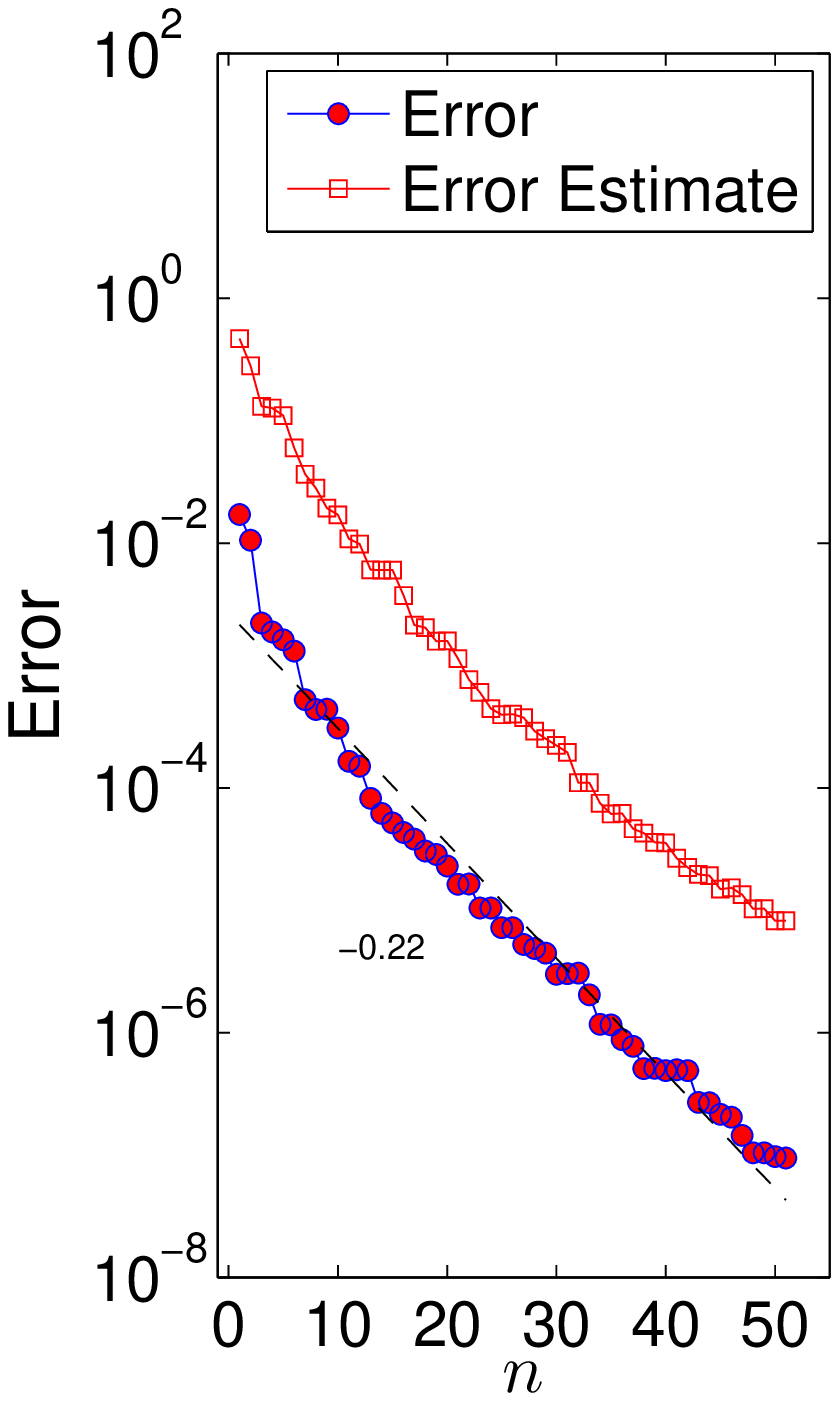} \\
\caption{Selected parameter values for the RB space generation (left), and the convergence of the R$^2$BFM solutions (right). Plotted on top is 
for the 3D problem \eqref{eq:awave3d} and at the bottom is for \eqref{eq:diff3d}. }
\label{fig:3DconvRB}
\end{figure} 

\subsubsection{\rrbm{} efficiency: Computational time}
\label{sec:comptime}
In this section, we report, in Table \ref{tab:comptime}, the computational time for the different stages of the method and the speedup of the \rrbm{}.  All computations are carried out using \textsc{MATLAB2013.b} on a Mac Pro workstation with 12 GB ECC memory and an eight core 2.8 GHz Intel Xeon processor. 
In Table \ref{tab:comptime}, $n$ is the number of reduced bases we use for the comparison between the reduced 
solver and the full solver (with a particular $\calN$ and $n_\text{loc}$). They are determined by having the accuracy of the two solvers roughly comparable (around $10^{-4}$). 
The former is located by referring to Figures \ref{fig:RBsolutions} and \ref{fig:3DconvRB}, and the later by Figures 
\ref{fig:algcov_awdiff} and \ref{fig:geocov_awdiff} with the specific values of $\calN$ and $n_\text{loc}$.

$\tau_{\textnormal{offline}}$ represents the offline computational time (in seconds) to compute $n$ basis snapshots along with associated reduced-order operators. 
We point out that $\tau_{\textnormal{offline}}$ does not include the time spent on the calculation of 
the stability constant \eqref{eq:beta}. In this paper, we calculate directly these constants which takes time 
that is comparable (2D) or more than (3D) $\tau_{\textnormal{offline}}$. 
However, we have tested locating a lower bound of the stability constant by the natural-norm successive constraint method \cite{HKCHP}. 
This algorithm cuts the time for the stability constant calculation by at least $75\%$ in all cases.

$\bar{\tau}_{\textnormal{ts}}$ is the average computational time to solve the PDE with the truth RBF-FD solver. 
$\bar{\tau}_{\textnormal{rb}}$ is that by using the 
\rrbm{} solver with $n$ basis functions. We observe that, in all cases, the R$^2$BFM provides speedup factors of more than $55$.
We note that the moderate speedup (compared to the usually-reported RBM speedups of $\mathcal{O}(100)$) is due to the specific implementation of our \textsc{MATLAB} code. In particular, the online assembling time is in the order of $Q_a n^2$ which should be negligible in comparison to the time devoted to solving for the reduced solution which is $O(n^3)$. 

Unfortunately, this is not the case in our implementation due to the use of the cell data structures and multi-dimentional arrays in \textsc{MATLAB}. In fact, if we exclude the assembling time in our calculation, we {\em recover} the usual speedup of 2 to 3 orders of magnitude for these type of problems. This speedup calculation (where online computational time does not include operator assembly time) is shown in Figure \ref{fig:comptime} for a large number of different $n$ and $\calN$.

\begin{table}[ht]
\begin{tabular}{|c|c|c|c|c|c|}
\hline
Problem & $n$  & $\tau_{\textnormal{offline}}$ (s)& $\bar{\tau}_{\textnormal{ts}}$ (s)  &$\bar{\tau}_{\textnormal{rb}}$ (s) & Speed up $\left( \frac{\bar{\tau}_{\textnormal{ts}}}{\bar{\tau}_{\textnormal{rb}}} \right)$ \\
\hline
1 (2D) & 12 & 190.28 & 0.0603576 & 0.00108305 & 55.7290 \\
\hline
2 (2D) & 12  & 111.65 & 0.0690138 & 0.00125392 & 55.0836 \\
\hline
1 (3D) &   6  & 145.89 & 0.1513660 & 0.00168018 & 90.0895 \\
\hline
2 (3D) & 12  & 221.38 & 0.1463470 & 0.00183657 & 79.6850 \\
\hline
\end{tabular}
\caption{Computational time and the corresponding speedup:.}
\label{tab:comptime}
\end{table}

\begin{figure}[htb]
\includegraphics[width=0.45\textwidth]{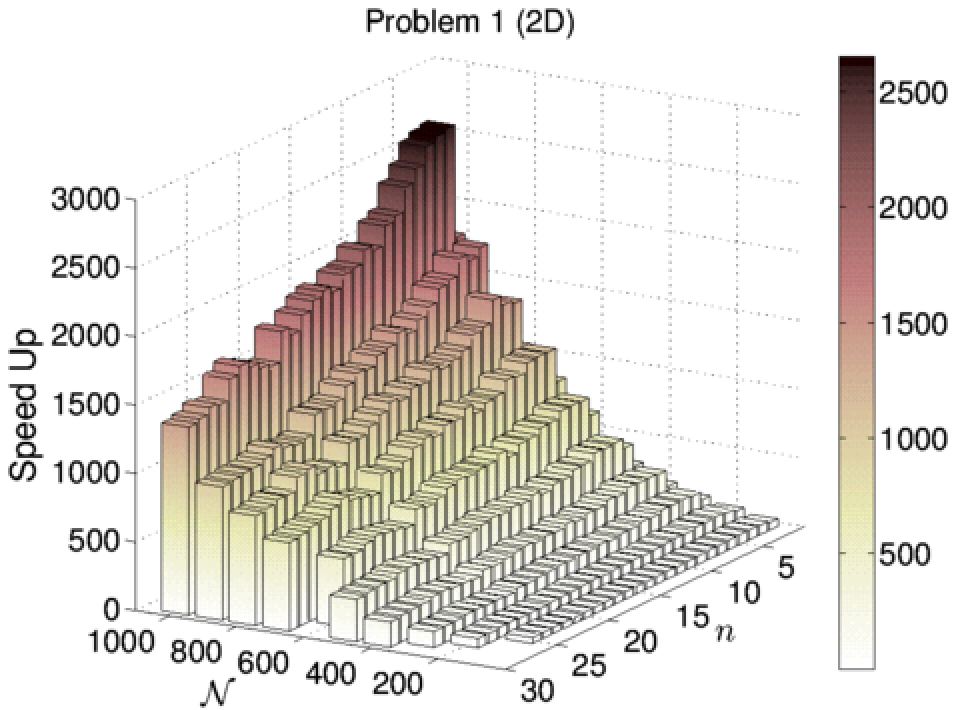}
\includegraphics[width=0.45\textwidth]{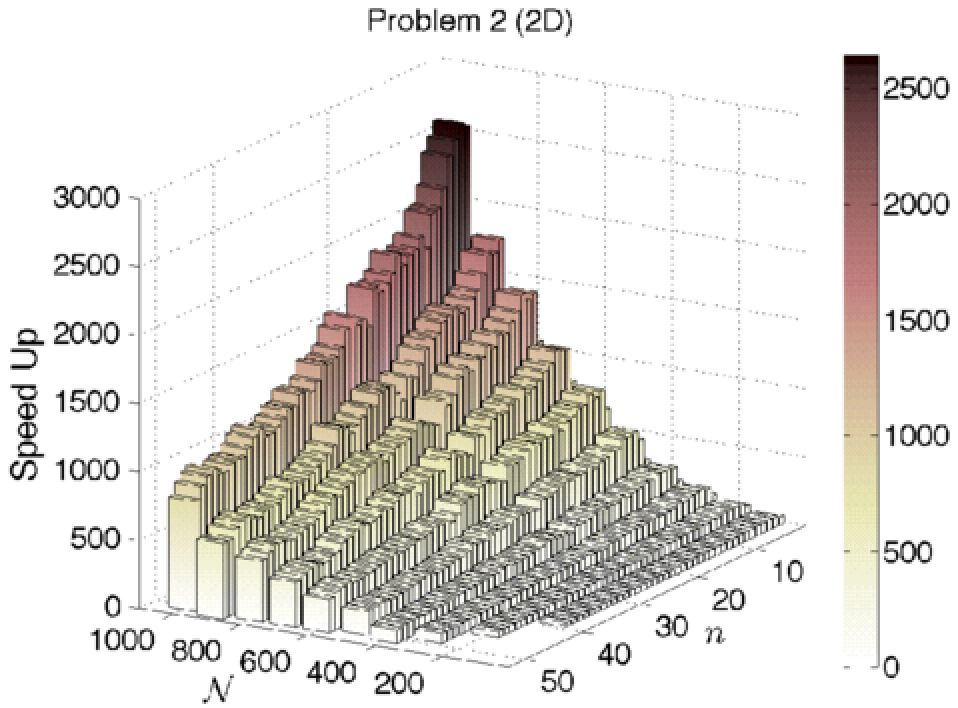} \\
\includegraphics[width=0.45\textwidth]{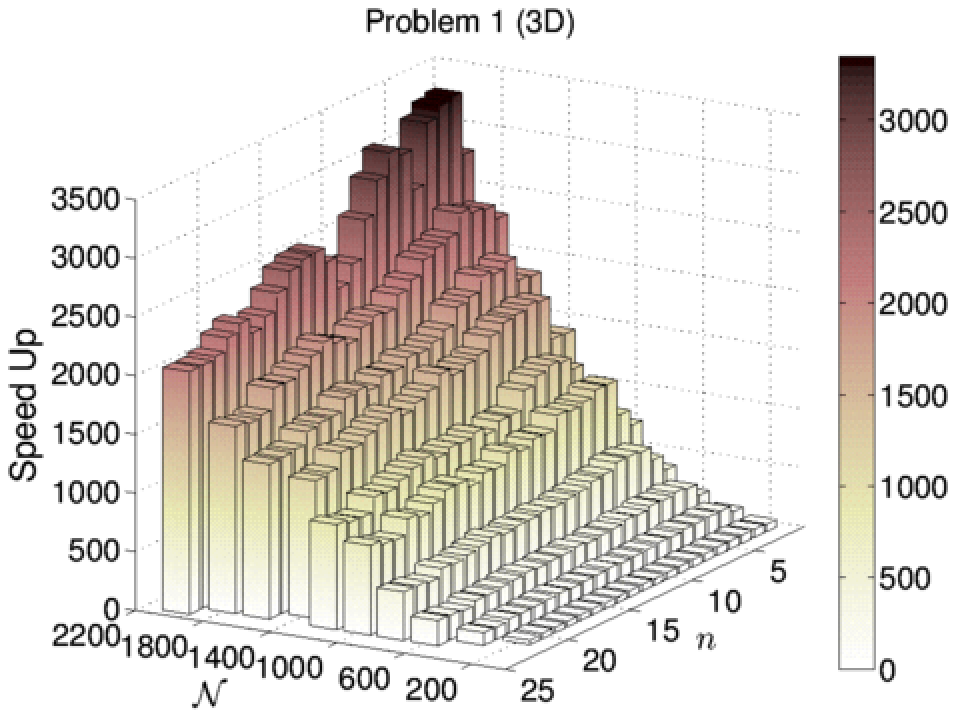}
\includegraphics[width=0.45\textwidth]{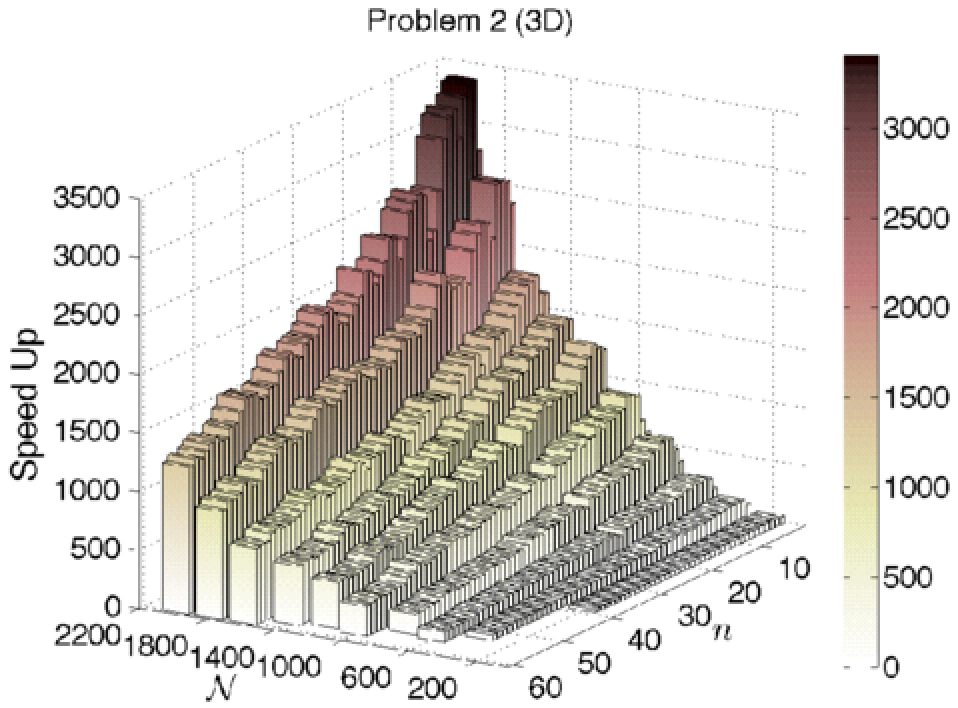} \\
\caption{Computational speedup excluding the online assembling time: shown on top are for the 2D problem \eqref{eq:awave2d} (left) and \eqref{eq:diff2d} (right), on the bottom are \eqref{eq:awave3d} (left) and \eqref{eq:diff3d} (right).}
\label{fig:comptime}
\end{figure}

% Remark
\section{Conclusion}
\label{sec:conclude}

Partial differential equations that have parametric dependence are challenging problems in scientific computing. The reduced-basis method efficiently handles these problems, even in the collocation setting for pseudospectral approximations. However, when the problem has difficulty compounded by an irregular geometry, standard pseudospectral collocation methods (e.g. Chebyshev, Fourier) cannot be directly applied. 

We have addressed this problem by applying a local radial basis function approximation methods to this situation. 
Due to their ability to approximate on irregular geometries, RBF methods provide excellent candidates for a collocation approximation. In particular we have employed a finite-difference version of RBF methods that uses local stencils to form differential operator approximations. The result is an $h p$-adaptive-like method on irregular geometries with local, hence efficient, operators.

We use the RBF-FD method as the truth approximation in a reduced basis collocation method, resulting in the Reduced Radial Basis Function Method. We have shown via extensive tests that this \rrbm{} algorithm can efficiently solve parametric problems on irregular geometries, effectively combining the strengths of both RBM and RBF algorithms.

\bibliographystyle{amsplain}

\providecommand{\bysame}{\leavevmode\hbox to3em{\hrulefill}\thinspace}
\providecommand{\MR}{\relax\ifhmode\unskip\space\fi MR }
% \MRhref is called by the amsart/book/proc definition of \MR.
\providecommand{\MRhref}[2]{%
  \href{http://www.ams.org/mathscinet-getitem?mr=#1}{#2}
}
\providecommand{\href}[2]{#2}

\end{document}